\documentclass[a4paper,12pt]{article}
\usepackage[T2A]{fontenc}                      
\usepackage[cp1251]{inputenc}           
\usepackage[all]{xy}
\usepackage{amssymb}
\usepackage{cite}
\usepackage{cmap}
\usepackage{latexsym}
\usepackage{enumerate}
\usepackage{amsmath, amsthm, amscd, amsfonts, amssymb, graphicx, color}
\usepackage[left=2.5cm,right=2.5cm,top=2cm,bottom=2cm]{geometry}
\usepackage{indentfirst}
\usepackage{array}
\usepackage{bm}
\usepackage{float}

\usepackage{wrapfig}

\newtheorem{corollary}{Corollary}
\newtheorem{conjecture}{Conjecture}
\newtheorem{proposition}{Proposition}

\newtheorem{theorem}{Theorem}
\newtheorem{lemma}{Lemma}
\newtheorem{definition}{Definition}

\begin{document}	
\title{From roots to paths: graphs simultaneously irregular with respect to rooted and ordinary paths}
\date{}
\author
	{Tatiana Dovzhenok\thanks{Research Laboratory "Mathematics of Hybrid Intelligence Systems", 	Francisk Skorina Gomel State University, 246028,  Gomel,  Belarus. 		
	E-mail: \texttt{t.dovzhenok@mail.ru}}}
\maketitle

\begin{abstract}	
	Let $P_n$ denote a path on $n$ vertices. A simple finite graph $G$ is called $P_n$-irregular if any two distinct vertices of $G$ belong to a different number of subgraphs of $G$ isomorphic to $P_n$. Alternatively, for a fixed vertex $r$ of $P_n$ (the root), $G$ is called $(P_n)_r$-irregular if any two distinct vertices of $G$ act as the root $r$ in a different number of subgraphs of $G$ isomorphic to $P_n$. 
	This paper proves that for each integer $k \geq 4$, there exists an infinite family of graphs that are simultaneously $P_n$-irregular and $(P_n)_r$-irregular for every integer $n$ satisfying $4 \leq n \leq k$ and every root $r$ of $P_n$. For the path $P_3$, we observe that no nontrivial $(P_3)_r$-irregular graphs exist if $r$ is the central vertex. In contrast, if $r$ is an end-vertex of $P_3$, an infinite collection of graphs is constructed that are both $P_3$-irregular and $(P_3)_r$-irregular. In particular, these results confirm the Strong Conjecture about $F$-irregular graphs for the case where $F$ is a path $P_n$.
	
\end{abstract}

\begin{flushright} 
	\itshape Dedicated with deep gratitude \\ 
	to Maestro Yakov Samuilovich Konstantinovsky. \hspace{-2ex}
	\vspace{0.4cm} 
	
	We both cultivate trees: I in the field of graph theory, \\ 
	and he in the garden of life. 
\end{flushright}
	
\section{Introduction}
Throughout this paper, we consider only finite, simple, and undirected graphs. 
For a graph $G$, let $V(G)$ and $E(G)$ be its vertex and edge sets, respectively. 
The order of $G$, denoted by $|G|$, is the number of its vertices; 
we say that $G$ is nontrivial if $|G| \geq 2$. 
The neighborhood of a vertex $v$ in $G$, denoted by $N_G(v)$, is the set of vertices adjacent to $v$, 
and its closed neighborhood is defined as $N_G[v] = N_G(v) \cup \{v\}$. 
The degree of $v$ is $\deg_G(v) = |N_G(v)|$, while $\delta(G)$ stands for the minimum degree of $G$.

We use standard notation for common graph families of order $n$: $P_n$ for the path, $C_n$ for the cycle, $K_n$ for the complete graph, and $K_{1,n-1}$ for the star. Unless otherwise stated, all parameters in this work are assumed to be non-negative integers. 

To ensure precision, particularly when vertices are labeled numerically, 
we adopt a specific convention for representing paths. 

Formally, for $n \geq 1$, a path $P_n$ is a graph with vertex set $\{v_1, v_2, \dots, v_n\}$ such that two vertices $v_i$ and $v_j$ are adjacent if and only if $|i - j| = 1$. We denote the path by 
listing its vertices in the order of their occurrence:
\[
P_n = (v_1, v_2, \dots, v_n).
\]
This notation is naturally invariant with respect to the direction of 
traversal, meaning that $(v_1, v_2, \dots, v_n)$ and 
$(v_n, v_{n-1}, \dots, v_1)$ represent the same object. By convention, 
we extend the definition to $n=0$ to account for the empty graph. 
Under this unified approach, an edge with end-vertices~$u$ and~$v$ is viewed as a path of order~$2$ and written as $(u, v)$. To prevent any ambiguity, ordered tuples are denoted by angle brackets, such as the ordered pair $\langle a, b \rangle$ of elements~$a$ and~$b$.
		
The primary subject of our study is $F$-irregularity. 
Introduced by Chartrand, Holbert, Oellermann, and Swart~\cite{r1} as a generalization of the classical vertex degree, this concept was originally presented alongside an elegant yet deceptively simple conjecture.
\begin{definition}
	For any two graphs $F$ and $G$, the \textbf{\bm{$F$}-degree} of a vertex $v$ in $G$, denoted by $F\text{-}\deg_G(v)$, is the number of subgraphs of $G$ isomorphic to $F$ that contain $v$. A graph $G$ is called \textbf{\bm{$F$}-irregular} if all its vertices have pairwise distinct $F$-degrees; that is, 
	\[ 
	F\text{-}\deg_G(u) \neq F\text{-}\deg_G(v) \quad \text{for all distinct } u, v \in V(G). 
	\] 
\end{definition}	

\begin{conjecture}[\cite{r1}, 1987] \label{con1}
	For every connected graph $F$ of order at least $3$, there exists a nontrivial $F$-irregular graph. 
\end{conjecture} 		 

Here, a fascinating paradox arises. While no nontrivial graph can be $K_2$-irregular (i.e., with pairwise distinct vertex degrees; see, e.g.,~\cite{r2}), this impossibility appears to be a consequence of the structural simplicity of~$K_2$. Conjecture~\ref{con1} suggests that once the graph~$F$ attains sufficient complexity ($|F| \geq 3$), these classical constraints vanish, rendering $F$-irregularity potentially universal.

The validity of Conjecture~\ref{con1} has been established for several classes of graphs, including stars, complete graphs~\cite{r1}, and paths~\cite{r3}. For each $2$-connected graph~$F$ with minimum degree $\delta(F)=2$ (in particular, for cycles), Dovzhenok, Filuta, and Chuhai~\cite{r4} constructed infinite families of $F$-irregular graphs and proposed a more ambitious conjecture.

\begin{conjecture}[Strong Conjecture about $F$-irregular graphs, \cite{r4}, 2024] \label{con2}
	For each connected graph~$F$ of order $|F| \geq 3$, there exist infinitely many $F$-irregular graphs.
\end{conjecture}

Conjecture~\ref{con2} has been fully settled for graphs $F$ of diameter~$2$ by Dovzhenok~\cite{r5}. 
It also holds for the paths~$P_3$ and~$P_4$ in view of~\cite{r6} and~\cite{r3}, respectively.

In this paper, we confirm Conjecture~\ref{con2} for the entire family of paths~$P_n$ with~$n \geq 3$. 
We also investigate the concept of rooted $F$-degrees introduced in~\cite{gtwa}, focusing on the case where~$F = P_n$. 
In contrast to the classical $F$-degree, which counts all copies of~$F$ containing a given vertex, 
the rooted variant counts only those where the vertex plays a designated role. To our knowledge, this work provides the first systematic study of graphs exhibiting rooted $F$-irregularity.
 
\begin{definition}
	Let $F$ be a graph with a distinguished vertex $r \in V(F)$, referred to as the root. For any vertex $v \in V(G)$, its {\boldmath\textbf{rooted $F$-degree}} (also termed {\boldmath\textbf{$(F)_r$-degree}}), denoted by $(F)_r\text{-}\deg_G(v)$, is the number of subgraphs of $G$ isomorphic to $F$ where $v$ corresponds to the root~$r$. More formally, this value equals the number of subgraphs $H \subseteq G$ for which there exists an isomorphism $\phi \colon F \to H$ satisfying $\phi(r) = v$. 
	A graph $G$ is {\boldmath\textbf{$(F)_r$-irregular}} if all its vertices have pairwise distinct $(F)_r$-degrees; that is, 
	\[ (F)_r\text{-}\deg_G(u) \neq (F)_r\text{-}\deg_G(v) \quad \text{for all distinct } u, v \in V(G). \]
\end{definition}
	
Our central aim is the construction of graphs that are simultaneously $P_n$- and $(P_n)_r$-irregular across a wide range of distinct pairs~$\langle n, r \rangle$. This phenomenon constitutes a \textit{super-concentration of irregularities}, wherein a single structure satisfies a large system of independent asymmetry constraints. Our first main result establishes both the existence and the abundance of such highly non-uniform graphs.

\begin{theorem} \label{thm1}
	For each integer $k \geq 4$, there exists an infinite family of graphs that are simultaneously $P_n\text{-}$irregular and $(P_n)_r\text{-}$irregular for all integers $n \in \{4, 5, \dots, k\}$ and every root~$r$ of~$P_n$.
\end{theorem}

Regarding the path~$P_3$, it is straightforward to observe that no nontrivial $(P_3)_r\text{-}$irregular graphs exist if $r$ is the central vertex. Conversely, when $r$ is an end-vertex, an analogue of Theorem~\ref{thm1} holds.

\begin{theorem} \label{thm2}
	There exists an infinite family of graphs that are simultaneously $P_3$\nobreakdash-irregular and $(P_3)_r$\nobreakdash-irregular, where $r$ is an end-vertex of~$P_3$.
\end{theorem}

For brevity, we identify the root $r = v_i$ of the path $P_n = (v_1, v_2, \dots, v_n)$ directly with its index~$i$. By exploiting symmetry, we can, without loss of generality, assume that $r \in \{1, 2, \dots, \lceil n/2 \rceil\}$.

The paper is organized as follows. Section~2 introduces a methodology for comparing $(P_n)_r$-degrees. Section~3 investigates $P_3$- and $(P_3)_r$-irregularity, providing the proof of Theorem~\ref{thm2}. Section~4 addresses the case $n \ge 4$ and presents a constructive proof of Theorem~\ref{thm1}. Finally, Section~5 summarizes our findings and proposes an analogue of the Strong Conjecture for $(F)_r$-irregular graphs.

\section{Methodology for comparing vertex $(P_n)_r$-degrees}

For a vertex $u$ of a graph $G$ and a root $r$ of the path $P_n$, let $\mathcal{P}_n(G, u, r)$ denote the set of all subgraphs of $G$ isomorphic to $P_n$ containing $u$, in which $u$ corresponds to the root~$r$. By the definition of $(P_n)_r$-degree, we have
\[
(P_n)_r\text{-}\deg_G(u) = |\mathcal{P}_n(G, u, r)|.
\]

To compare the $(P_n)_r$-degrees of two distinct vertices $u_1, u_2 \in V(G)$, we identify a subset of edges $\mathcal{E} \subset E(G)$ such that $u_1$ and $u_2$ are symmetric in the subgraph $G_1 = G \setminus \mathcal{E}$. This symmetry implies the existence of an automorphism of $G_1$ that maps $u_1$ to $u_2$. In our constructions, this condition is guaranteed by ensuring that $u_1$ and $u_2$ have either identical neighborhoods, $N_{G_1}(u_1) = N_{G_1}(u_2)$, or identical closed neighborhoods, $N_{G_1}[u_1] = N_{G_1}[u_2]$. Consequently, 
\[
(P_n)_r\text{-}\deg_{G_1}(u_1) = (P_n)_r\text{-}\deg_{G_1}(u_2),
\]
and the difference between the $(P_n)_r$-degrees of $u_1$ and $u_2$ in $G$ satisfies
\begin{equation} \label{eq1}
	(P_n)_r\text{-}\deg_G(u_1) - (P_n)_r\text{-}\deg_G(u_2) = |\mathcal{A}| - |\mathcal{B}|,
\end{equation}
where the sets $\mathcal{A}$ and $\mathcal{B}$ are defined as 
\begin{align*}
	\mathcal{A} &= \{F \in \mathcal{P}_n(G, u_1, r) \mid E(F) \cap \mathcal{E} \neq \emptyset\}, \\
	\mathcal{B} &= \{H \in \mathcal{P}_n(G, u_2, r) \mid E(H) \cap \mathcal{E} \neq \emptyset\}.
\end{align*}

Thus, comparing $(P_n)_r$-degrees reduces to evaluating the relative cardinalities of $\mathcal{A}$ and $\mathcal{B}$. For small $n$, this is achieved by direct calculation or simple estimates. For larger $n$, we employ a hybrid approach by partitioning $\mathcal{A}$ and $\mathcal{B}$ into subfamilies whose sizes are compared either through injective mappings or combinatorial estimates involving binomial coefficients. These combinatorial bounds are valid under certain constraints, primarily when $|G|$ is sufficiently large.
	
\section{The Case $n=3$}
We begin our study by considering the simplest case, $n=3$. We examine two scenarios: $r=2$, where the root is the central vertex of $P_3$, and $r=1$, where it is an end-vertex.

\subsection{$(P_3)_2$-irregular graphs}

\begin{proposition}  \label{p1}
	No nontrivial graph is $(P_3)_2$-irregular.
\end{proposition}	

\begin{proof}
	Let $G$ be an arbitrary nontrivial graph and let $v \in V(G)$. Every pair of distinct vertices $\{u_1, u_2\} \subseteq N_G(v)$ uniquely determines a subgraph isomorphic to $P_3$ centered at $v$, namely the path $(u_1, v, u_2)$. Conversely, any such path centered at $v$ is uniquely determined by its end-vertices in $N_G(v)$. This bijective correspondence yields
	\begin{equation} \label{eq2}
		(P_3)_2\text{-}\deg_G(v) = \binom{\deg_G(v)}{2}.
	\end{equation}
	Since $G$ contains at least two vertices of the same degree, by~\eqref{eq2} their $(P_3)_2$-degrees must also be equal. Consequently, $G$ cannot be $(P_3)_2$-irregular.
\end{proof}

\subsection{$(P_3)_1$-irregular graphs}
In contrast to the case $r=2$, the $(P_3)_1$\nobreakdash-degree of a vertex $v$ in a graph $G$ depends not only on $\deg_G(v)$ but also on the degrees of its neighbors in $G$. In this subsection, we construct an infinite family of graphs that are simultaneously $P_3$\nobreakdash-irregular and $(P_3)_1$\nobreakdash-irregular.
		
\subsubsection{Graph $\mathcal{T}_{2m+1}$}
Throughout this subsection, we let $m$ be an arbitrary but fixed integer such that $m \ge 6$.
\begin{definition}
	The graph $\mathcal{T}_{2m+1}$ of order $2m+1$ is defined by the vertex set
	\[
	V(\mathcal{T}_{2m+1}) = V_1 \cup V_2 \cup \{2m+1\},
	\]
	where $V_1 = \{1, 2, \dots, m\}$ and $V_2 = \{m+1, m+2, \dots, 2m\}$, and the edge set
\begin{align*}
	E(\mathcal{T}_{2m+1}) &= \{(i, j) \mid i, j \in V_1, \ i \neq j\} \\
	&\quad \cup \{(i, j) \mid i \in V_1, \ j \in V_2, \ j - i \leq m\} \\
	&\quad \cup \{(1, 2m-1), (1, 2m), (2m, 2m+1)\}.
\end{align*}
\end{definition}

To visualize the graph $\mathcal{T}_{2m+1}$, we arrange its vertices on two horizontal levels (see Fig.~1). The vertices of $V_1$ and $V_2 \cup \{2m+1\}$ are placed on the upper and lower levels, respectively, in increasing order of their indices from left to right. In particular, each vertex $i \in V_1$ is positioned directly above the vertex $i+m \in V_2$.
	
In this layout, the vertices on the upper level induce a clique. Each vertex $j \in V_2$ is adjacent to the vertex $j-m$ located directly above it, as well as to all vertices on the upper level to the right of $j-m$. The structure is completed by the three specific edges: $(1, 2m-1)$, $(1, 2m)$, and $(2m, 2m+1)$. 
\medskip

For each $v \in V(\mathcal{T}_{2m+1})$, let $X_v$ denote the $(P_3)_1$\nobreakdash-degree of vertex $v$ in $\mathcal{T}_{2m+1}$.

\begin{figure}[ht]
	\centering
	\includegraphics[width=13.5cm]{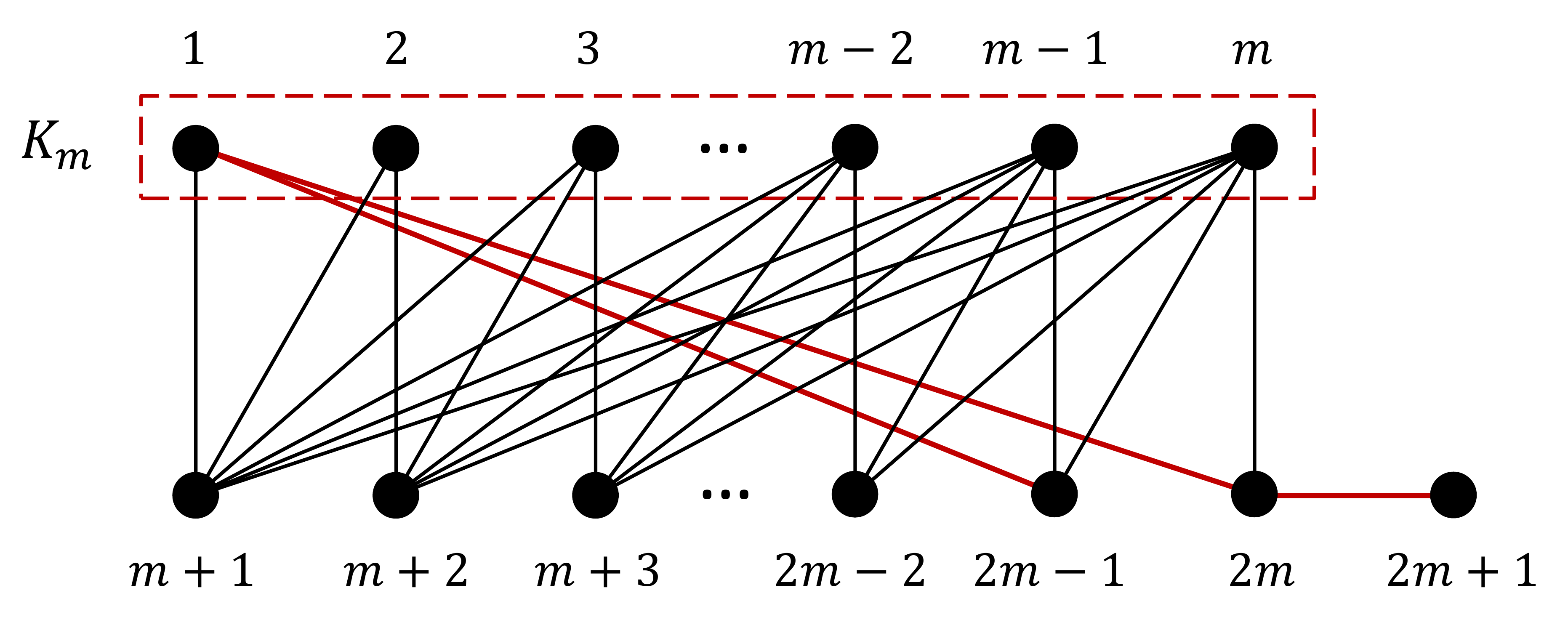}
	\caption{Graph $\mathcal{T}_{2m+1}$.}
	\label{fig1}
\end{figure}

\medskip

In Lemmas~\ref{lemma1}--\ref{lemma7}, we compare the $(P_3)_1$-degrees of vertices in the graph~$\mathcal{T}_{2m+1}$. Although these values can be obtained by direct calculation, the methodology described in Section~2 enables a more elegant and straightforward comparison in all cases, with the sole exception of Lemma~\ref{lemma1}.

\begin{lemma} \label{lemma1}
	$X_{2m} > X_{2m+1}$.
\end{lemma}

\begin{proof}
	Observe that $X_{2m+1} = 2$, as there are exactly two paths isomorphic to $P_3$ in $\mathcal{T}_{2m+1}$ originating at $2m+1$, namely $(2m+1, 2m, 1)$ and $(2m+1, 2m, m)$. On the other hand, there are at least three such paths originating at $2m$, for instance, $(2m, m, 1)$, $(2m, m, 2)$, and $(2m, m, 3)$. This implies $X_{2m} \ge 3$, which immediately yields the desired inequality.
\end{proof}

\begin{lemma} \label{lemma2}
	$X_{2m-1} > X_{2m}$.
\end{lemma}

\begin{proof}
	Let $\mathcal{E}_1 = \{(m-1, 2m-1), (2m, 2m+1)\}$ and $G_1 = \mathcal{T}_{2m+1} \setminus \mathcal{E}_1$. The vertices $2m-1$ and $2m$ are symmetric in $G_1$ since their neighborhoods are identical:
	\[
	N_{G_1}(2m-1) = N_{G_1}(2m) = \{1, m\}.
	\]
	By \eqref{eq1}, the difference between the $(P_3)_1$-degrees of $2m-1$ and $2m$ in $\mathcal{T}_{2m+1}$ simplifies to
	\begin{align*}
		X_{2m-1} - X_{2m} &= |\mathcal{A}| - |\mathcal{B}|, \\
		\mathcal{A} &= \{F \in \mathcal{P}_3(\mathcal{T}_{2m+1}, 2m-1, 1) \mid E(F) \cap \mathcal{E}_1 \neq \emptyset \}, \\
		\mathcal{B} &= \{H \in \mathcal{P}_3(\mathcal{T}_{2m+1}, 2m, 1) \mid E(H) \cap \mathcal{E}_1 \neq \emptyset \}.
	\end{align*}
	Note that $\mathcal{B}$ is empty because no path isomorphic to $P_3$ in $\mathcal{T}_{2m+1}$ originating at $2m$ contains an edge from $\mathcal{E}_1$. In contrast, $\mathcal{A}$ is non-empty, as it contains the path $(2m-1, m-1, 1)$. Thus, $|\mathcal{A}| > |\mathcal{B}|$, which implies $X_{2m-1} > X_{2m}$.
\end{proof}

\begin{lemma} \label{lemma3}	
	$X_{2m-2} > X_{2m-1}$.
\end{lemma}

\begin{proof}
	Let $\mathcal{E}_2 = \{(1, 2m-1), (m-2, 2m-2)\}$ and $G_2 = \mathcal{T}_{2m+1} \setminus \mathcal{E}_2$. The vertices $2m-2$ and $2m-1$ possess identical neighborhoods in $G_2$:
	\[
	N_{G_2}(2m-2) = N_{G_2}(2m-1) = \{m-1, m\},
	\]
	which renders them symmetric in this subgraph. Consequently, applying \eqref{eq1} yields
	\begin{align*}
		X_{2m-2} - X_{2m-1} &= |\mathcal{A}| - |\mathcal{B}|, \\
		\mathcal{A} &= \{F \in \mathcal{P}_3(\mathcal{T}_{2m+1}, 2m-2, 1) \mid E(F) \cap \mathcal{E}_2 \neq \emptyset \}, \\
		\mathcal{B} &= \{H \in \mathcal{P}_3(\mathcal{T}_{2m+1}, 2m-1, 1) \mid E(H) \cap \mathcal{E}_2 \neq \emptyset \}.
	\end{align*}
	Explicitly, these sets can be expressed as
	\begin{align*}
		\mathcal{A} &= \{(2m-2, m-2, u) \mid u \in \{1, 2, \dots, 2m-3\} \setminus \{m-2\}\}, \\
		\mathcal{B} &= \{(2m-1, 1, v) \mid v \in \{2, 3, \dots, m+1\} \cup \{2m\}\}.
	\end{align*}	
	Since $|\mathcal{A}| = 2m-4$ and $|\mathcal{B}| = m+1$, the difference $|\mathcal{A}| - |\mathcal{B}| = m-5$ is strictly positive, as $m \geq 6$. Hence, $X_{2m-2} > X_{2m-1}$.
\end{proof}

\begin{lemma} \label{lemma4} 
	For each $i \in \{m+1, m+2, \dots, 2m-3\}$, the following inequality holds: 
	\begin{equation*}
		X_{i} > X_{i+1}.
	\end{equation*}
\end{lemma}

\begin{proof}
	Fix $i \in \{m+1, m+2, \dots, 2m-3\}$. Let $\mathcal{E}_3 = \{(i-m, i)\}$ and $G_3 = \mathcal{T}_{2m+1} \setminus \mathcal{E}_3$ (see Fig.~2). In $G_3$, the vertices $i$ and $i+1$ are symmetric since their neighborhoods coincide:
	\[
	N_{G_3}(i) = N_{G_3}(i+1) = \{i-m+1, i-m+2, \dots, m\}.
	\]
	
	\begin{figure}[ht]
		\centering
		\includegraphics[width=13.5cm]{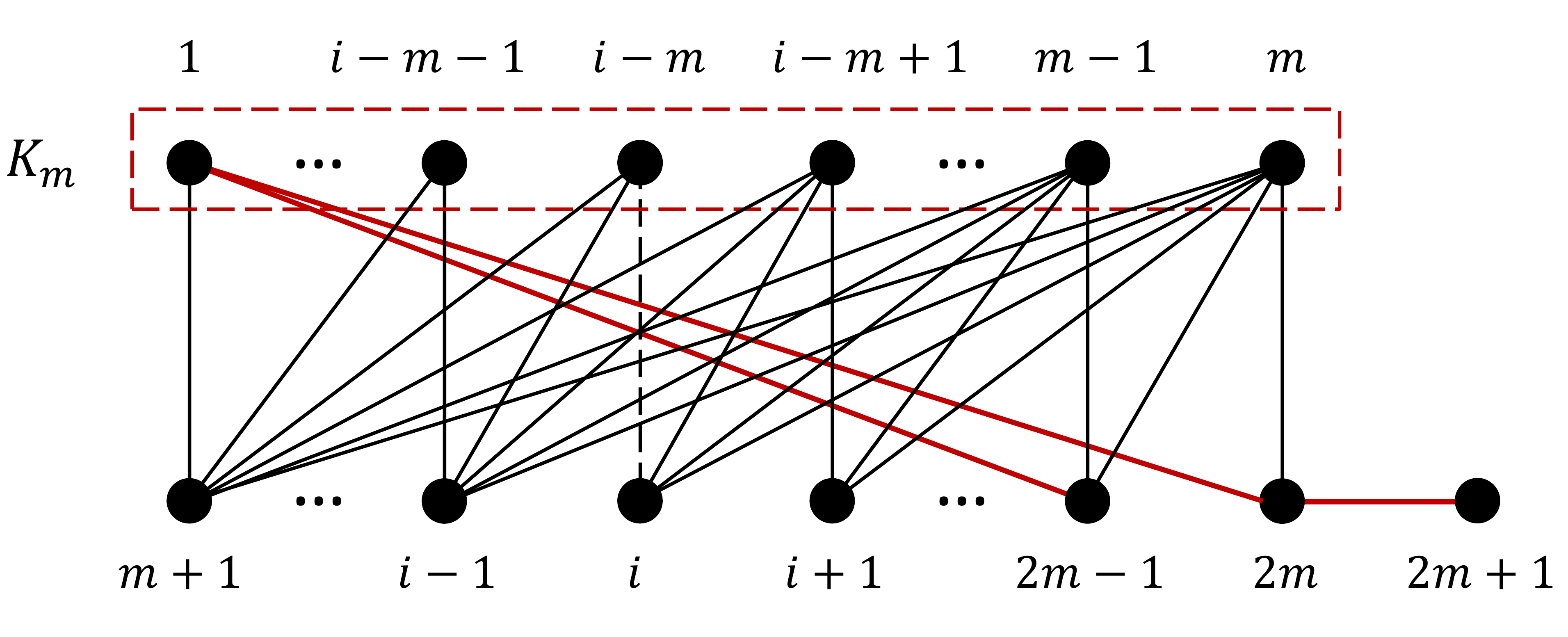}
		\caption{Graph $G_3 = \mathcal{T}_{2m+1} \setminus \{(i-m, i)\}$.}
		\label{fig2}
	\end{figure}
	
\noindent By \eqref{eq1}, we obtain
	\begin{align*}
		X_{i} - X_{i+1} &= |\mathcal{A}| - |\mathcal{B}|, \\
		\mathcal{A} &= \{F \in \mathcal{P}_3(\mathcal{T}_{2m+1}, i, 1) \mid (i-m, i) \in E(F) \}, \\
		\mathcal{B} &= \{H \in \mathcal{P}_3(\mathcal{T}_{2m+1}, i+1, 1) \mid (i-m, i) \in E(H)\}.
	\end{align*}	
	It is clear that $\mathcal{B} = \emptyset$, as no path isomorphic to $P_3$ in $\mathcal{T}_{2m+1}$ originating at $i+1$ can include the edge $(i-m, i)$. On the other hand, the existence of the path $(i, i-m, m)$ ensures that $\mathcal{A}$ is non-empty. These observations show that $|\mathcal{A}| > |\mathcal{B}|$, and thus $X_{i} > X_{i+1}$.
\end{proof}

\begin{lemma} \label{lemma5}	
	$X_{1} > X_{m+1}$.
\end{lemma}
\begin{proof}
	Let $\mathcal{E}_4 = \{(1, 2m-1), (1, 2m)\}$ and $G_4 = \mathcal{T}_{2m+1} \setminus \mathcal{E}_4$. The vertices $1$ and $m+1$ are symmetric in $G_4$, as they share the same closed neighborhood: 
	\[
	N_{G_4}[1] = N_{G_4}[m+1] = \{1, 2, \dots, m+1\}.
	\]
	Consequently, equation~\eqref{eq1} yields
	\begin{align*}
		X_{1} - X_{m+1} &= |\mathcal{A}| - |\mathcal{B}|, \\
		\mathcal{A} &= \{F \in \mathcal{P}_3(\mathcal{T}_{2m+1}, 1, 1) \mid E(F) \cap \mathcal{E}_4 \neq \emptyset \}, \\
		\mathcal{B} &= \{H \in \mathcal{P}_3(\mathcal{T}_{2m+1}, m+1, 1) \mid E(H) \cap \mathcal{E}_4 \neq \emptyset \}.
	\end{align*}
	Direct verification confirms that $|\mathcal{A}| = 4$, with the explicit elements being
	\[
	\mathcal{A} = \{(1, 2m-1, m-1), (1, 2m-1, m), (1, 2m, m), (1, 2m, 2m+1)\}.
	\]
	Meanwhile, we find that $|\mathcal{B}| = 2$, namely
	\[
	\mathcal{B} = \{(m+1, 1, 2m-1), (m+1, 1, 2m)\}.
	\]	
	Since $|\mathcal{A}| > |\mathcal{B}|$, the desired inequality $X_{1} > X_{m+1}$ follows immediately.
\end{proof}
	
\begin{lemma} \label{lemma6}	
	$X_{2} > X_{1}$.
\end{lemma}

\begin{proof}
	Let $\mathcal{E}_5 = \{(1, 2m-1), (1, 2m), (2, m+2)\}$ and $G_5 = \mathcal{T}_{2m+1} \setminus \mathcal{E}_5$. The vertices $1$ and $2$ are symmetric in $G_5$, since they have identical closed neighborhoods:
	\[
	N_{G_5}[1] = N_{G_5}[2] = \{1, 2, \dots, m+1\}.
	\]
	By \eqref{eq1}, the difference $X_{2} - X_{1}$ can be expressed as
	\begin{align*}
		X_{2} - X_{1} &= |\mathcal{A}| - |\mathcal{B}|, \\
		\mathcal{A} &= \{F \in \mathcal{P}_3(\mathcal{T}_{2m+1}, 2, 1) \mid E(F) \cap \mathcal{E}_5 \neq \emptyset \}, \\
		\mathcal{B} &= \{H \in \mathcal{P}_3(\mathcal{T}_{2m+1}, 1, 1) \mid E(H) \cap \mathcal{E}_5 \neq \emptyset \}.
	\end{align*}
	As $m \geq 6$, the set $\mathcal{A}$ contains at least six paths, namely
	\[
	(2, 1, 2m-1), (2, 1, 2m), (2, m+2, 3), (2, m+2, 4), (2, m+2, 5), \text{ and } (2, m+2, 6).
	\]
	In turn, the set $\mathcal{B}$ consists of exactly five elements:
	\[
	\mathcal{B} = \{(1, 2, m+2), (1, 2m-1, m-1), (1, 2m-1, m), (1, 2m, m), (1, 2m, 2m+1)\}.
	\]
	Since $|\mathcal{A}| \geq 6 > |\mathcal{B}|$, we obtain $X_{2} > X_1$.	
\end{proof}

\begin{lemma} \label{lemma7}
	For each $i \in \{2, 3, \dots, m-1\}$, the following inequality holds: 
	\[
	X_{i+1} > X_{i}.
	\]
\end{lemma}

\begin{proof}
	Fix $i \in \{2, 3, \dots, m-1\}$ and let $G_6 = \mathcal{T}_{2m+1} \setminus \{(i+1, m+i+1)\}$ (see Fig.~3). The symmetry of $i$ and $i+1$ in $G_6$ is evident from their common closed neighborhood $\{1, 2, \dots, m+i\}$. 
	
	\begin{figure}[ht]
		\centering
		\includegraphics[width=13.5cm]{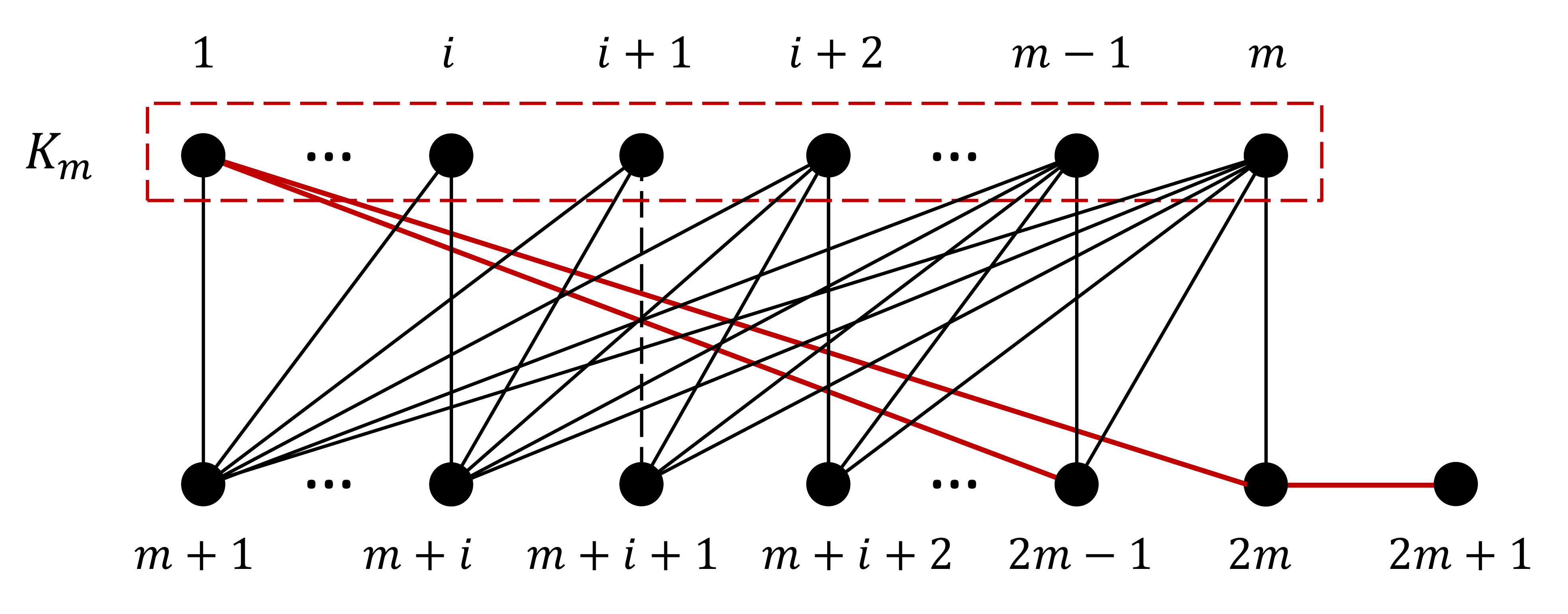}
		\caption{Graph $G_6 = \mathcal{T}_{2m+1} \setminus \{(i+1, m+i+1)\}$.}
		\label{fig3}
	\end{figure}
	
\noindent Under this symmetry, \eqref{eq1} implies that $X_{i+1} - X_i = |\mathcal{A}| - |\mathcal{B}|$, where
	\begin{align*}
		\mathcal{A} &= \{F \in \mathcal{P}_3(\mathcal{T}_{2m+1}, i+1, 1) \mid (i+1, m+i+1) \in E(F) \}, \\
		\mathcal{B} &= \{H \in \mathcal{P}_3(\mathcal{T}_{2m+1}, i, 1) \mid (i+1, m+i+1) \in E(H)\}.
	\end{align*}
It is easy to check that $\mathcal{B}$ consists only of the path $(i, i+1, m+i+1)$, hence $|\mathcal{B}| = 1$. In contrast, $|\mathcal{A}| \geq 2$ due to the presence of the following paths in $\mathcal{A}$:
	\begin{itemize}
		\item $(i+1, m+i+1, m-1)$ and $(i+1, m+i+1, m)$ for $2 \leq i \leq m-3$,
		\item $(m-1, 2m-1, 1)$ and $(m-1, 2m-1, m)$ for $i = m-2$,
		\item $(m, 2m, 1)$ and $(m, 2m, 2m+1)$ for $i = m-1$.
	\end{itemize}
	The resulting inequality $|\mathcal{A}| > |\mathcal{B}|$ confirms that $X_{i+1} > X_i$.		
\end{proof}

To provide the necessary foundation for the main result of Section~3, we now turn to ordinary $P_3$-degrees and establish the following chain of inequalities.

\begin{lemma} \label{lemma8}	 
	The $P_3$-degrees of vertices $1$, $2$, and $3$ in $\mathcal{T}_{2m+1}$ satisfy
	\begin{equation*}
		P_3\text{-}\deg_{\mathcal{T}_{2m+1}}(3) > P_3\text{-}\deg_{\mathcal{T}_{2m+1}}(1) > P_3\text{-}\deg_{\mathcal{T}_{2m+1}}(2).
	\end{equation*}
\end{lemma}

\begin{proof}
	As shown in the proof of Lemma~\ref{lemma6}, vertices $1$ and $2$ are symmetric in the subgraph $G_5 = \mathcal{T}_{2m+1} \setminus \mathcal{E}_5$, where 
	$\mathcal{E}_5 = \{(1, 2m-1), (1, 2m), (2, m+2)\}$.  
	Thus, they have the same $P_3$-degrees in $G_5$, which leads to the following relation for the original graph $\mathcal{T}_{2m+1}$:
	\[
	P_3\text{-}\deg_{\mathcal{T}_{2m+1}}(1) - P_3\text{-}\deg_{\mathcal{T}_{2m+1}}(2) = |\mathcal{C}| - |\mathcal{D}|,
	\]
	where $\mathcal{C}$ and $\mathcal{D}$ denote the sets of subgraphs in $\mathcal{T}_{2m+1}$ isomorphic to $P_3$ that contain at least one edge from $\mathcal{E}_5$ and pass through vertices $1$ and $2$, respectively.
	Clearly,  
	\begin{align*}
		\mathcal{C} &= \{(2m-1, 1, u) \mid u \in \{2, 3,\dots, m+1\}\cup \{2m\}\} \\
		&\quad \cup \{(2m, 1, u) \mid u \in \{2, 3,\dots, m+1\}\} \cup \{(1, 2, m+2)\} \\
		&\quad \cup \{(1, 2m-1, m-1), (1, 2m-1, m), (1, 2m, m), (1, 2m, 2m+1)\},
	\end{align*}
	whence $|\mathcal{C}| = 2m + 6$. In a similar manner, one finds 
	\begin{align*}
		\mathcal{D} &= \{(m+2, 2, u) \mid u \in \{3, 4,\dots, m+1\}\cup \{1\}\} \\
		&\quad \cup \{(2, m+2, u) \mid u \in \{3, 4,\dots, m\} \} \cup \{(2, 1, 2m-1), (2, 1, 2m)\},
	\end{align*}
	with $|\mathcal{D}| = 2m$. Since $|\mathcal{C}| > |\mathcal{D}|$, we obtain
	\[
	P_3\text{-}\deg_{\mathcal{T}_{2m+1}}(1) > P_3\text{-}\deg_{\mathcal{T}_{2m+1}}(2).
	\]
	
Similarly, the vertices $1$ and $3$ are symmetric in the subgraph $G_7 = \mathcal{T}_{2m+1} \setminus \mathcal{E}_7$, where 
$
\mathcal{E}_7 = \{(1, 2m-1), (1, 2m), (3, m+2), (3, m+3)\}.
$
Hence, 
\[
P_3\text{-}\deg_{\mathcal{T}_{2m+1}}(3) - P_3\text{-}\deg_{\mathcal{T}_{2m+1}}(1) = |\mathcal{X}| - |\mathcal{Y}|,
\]
where $\mathcal{X}$ and $\mathcal{Y}$ denote the sets of subgraphs of $\mathcal{T}_{2m+1}$ isomorphic to $P_3$ that contain at least one edge from $\mathcal{E}_7$ and pass through vertices 3 and 1, respectively. Note that $|\mathcal{Y}| = 2m + 7$, since $\mathcal{Y}$ is explicitly given by
\[
\mathcal{Y} = \{(1, 3, m+2), (1, 3, m+3)\} \cup (\mathcal{C} \setminus \{(1, 2, m+2)\}).
\]
An analysis of the set $\mathcal{X}$ yields the following explicit description:
\begin{align*}
	\mathcal{X} &= \{(m+2, 3, u) \mid u \in \{1, 2, \dots, m+3\} \setminus \{3, m+2\}\} \\
	&\quad \cup \{(3, m+2, u) \mid u \in \{2, 3, \dots, m\} \setminus \{3\} \} \\
	&\quad \cup \{(m+3, 3, u) \mid u \in \{1, 2, \dots, m+1\} \setminus \{3\}\} \\
	&\quad \cup \{(3, m+3, u) \mid u \in \{4, 5,\dots, m\}\} \cup \{(3, 1, 2m-1), (3, 1, 2m)\},
\end{align*}
whence $|\mathcal{X}| = 4m - 2$. Finally, since $m \ge 6$, the inequality $|\mathcal{X}| = 4m - 2 > 2m + 7 = |\mathcal{Y}|$ confirms that
\[
P_3\text{-}\deg_{\mathcal{T}_{2m+1}}(3) > P_3\text{-}\deg_{\mathcal{T}_{2m+1}}(1). \qedhere
\]
\end{proof}

\subsubsection{Proof of Theorem~\ref{thm2}}
Theorem~\ref{thm2} follows directly from Theorem~\ref{thm3} below. 

\begin{theorem} \label{thm3}
	For every integer $m \geq 6$, the graph $\mathcal{T}_{2m+1}$ is simultaneously $(P_3)_1$-irregular and $P_3$-irregular.
\end{theorem}

\begin{proof}
	Fix an arbitrary integer $m \geq 6$. According to Lemmas~\ref{lemma1}--\ref{lemma7}, all vertices of $\mathcal{T}_{2m+1}$ have distinct $(P_3)_1$-degrees, which satisfy the following chain of inequalities:
	\begin{equation} \label{eq3}
		X_{m} > X_{m-1} > \dots > X_{2} > X_{1} > X_{m+1} > X_{m+2} > \dots > X_{2m} > X_{2m+1}.
	\end{equation}
	Thus, $\mathcal{T}_{2m+1}$ is $(P_3)_1$-irregular. 
	
	Next, we establish the ordering of the degrees, the $(P_3)_2$-degrees, and finally the $P_3$\nobreakdash-degrees for all vertices in $\mathcal{T}_{2m+1}$ except for vertex 1. For brevity, the subscript $\mathcal{T}_{2m+1}$ will be omitted from the notation for all types of degrees throughout the remainder of this proof. From the structure of $\mathcal{T}_{2m+1}$ (see Fig.~1), it is straightforward to observe that	
	\begin{equation} \label{eq4}
		\begin{split}
			&\deg(m) > \deg(m-1) > \dots > \deg(3) > \deg(2), \\
			&\deg(2) > \deg(m+1), \\
			&\deg(m+1) \geq \deg(m+2) \geq \dots \geq \deg(2m+1) = 1.
		\end{split}
	\end{equation}
In view of relations \eqref{eq2} and \eqref{eq4}, the $(P_3)_2$-degrees are ranked as follows:
	\begin{equation} \label{eq5}
		\begin{split}
			&(P_3)_2\text{-}\deg(m) > (P_3)_2\text{-}\deg(m-1) > \dots > (P_3)_2\text{-}\deg(3) > (P_3)_2\text{-}\deg(2), \\
			&(P_3)_2\text{-}\deg(2) > (P_3)_2\text{-}\deg(m+1), \\
			&(P_3)_2\text{-}\deg(m+1) \geq (P_3)_2\text{-}\deg(m+2) \geq \dots \geq (P_3)_2\text{-}\deg(2m+1).
		\end{split}
	\end{equation}
	By combining \eqref{eq3}, \eqref{eq5}, and the fundamental relation
	\[
	P_3\text{-}\deg(v) = X_v + (P_3)_2\text{-}\deg(v), \quad v \in V(\mathcal{T}_{2m+1}),
	\]
	we obtain the ordering of the $P_3$-degrees:
	\begin{equation} \label{eq6}
		\begin{split}
			&P_3\text{-}\deg(m) > P_3\text{-}\deg(m-1) > \dots > P_3\text{-}\deg(3) > P_3\text{-}\deg(2), \\
			&P_3\text{-}\deg(2) > P_3\text{-}\deg(m+1), \\
			&P_3\text{-}\deg(m+1) > P_3\text{-}\deg(m+2) > \dots > P_3\text{-}\deg(2m+1).
		\end{split}
	\end{equation}
Finally, by Lemma~\ref{lemma8}, the $P_3$-degree of vertex 1 in $\mathcal{T}_{2m+1}$ lies strictly between the $P_3$\nobreakdash-degrees of vertices $2$ and $3$.
Combined with \eqref{eq6}, this implies that the $P_3$-degrees of all vertices in $\mathcal{T}_{2m+1}$ are pairwise distinct. Hence, $\mathcal{T}_{2m+1}$ is also $P_3$-irregular.
	\end{proof}

\section{On multiple irregularities with respect to paths and rooted paths of various orders $n \geq 4$}

This section presents the proof of Theorem~\ref{thm1}. Fix an integer $k \geq 4$. We consider an arbitrary $n \in \{4, 5, \dots, k\}$ and any choice of the root~$r$ for the path $P_n$. Subsection~4.1 introduces the graph $\mathcal{W}_{2m}$ of order $2m$ and establishes a series of lemmas comparing the $(P_n)_r$-degrees of its vertices under certain constraints on the parameter $m$. 
Subsection~4.2 formulates the $k$-condition on $m$, which ensures that all these constraints are simultaneously satisfied. 
Finally, Subsection~4.3 completes the proof of Theorem~\ref{thm1}.

\subsection{Graph $\mathcal{W}_{2m}$}
For the construction below, let $m$ be an arbitrary but fixed integer such that $m \ge 6$ and $m > k$; in particular, $m > n$.
\begin{definition} 
	The graph $\mathcal{W}_{2m}$ (see Fig.~4) is obtained from $\mathcal{T}_{2m+1}$ (see Fig.~1) as follows:
	\[
	V(\mathcal{W}_{2m}) = V(\mathcal{T}_{2m+1}) \setminus \{2m+1\}, \quad E(\mathcal{W}_{2m}) = E(\mathcal{T}_{2m+1}) \setminus \{(1, 2m), (2m, 2m+1)\}.
	\]
\end{definition}

\begin{figure}[ht]
	\centering
	\includegraphics[width=12cm]{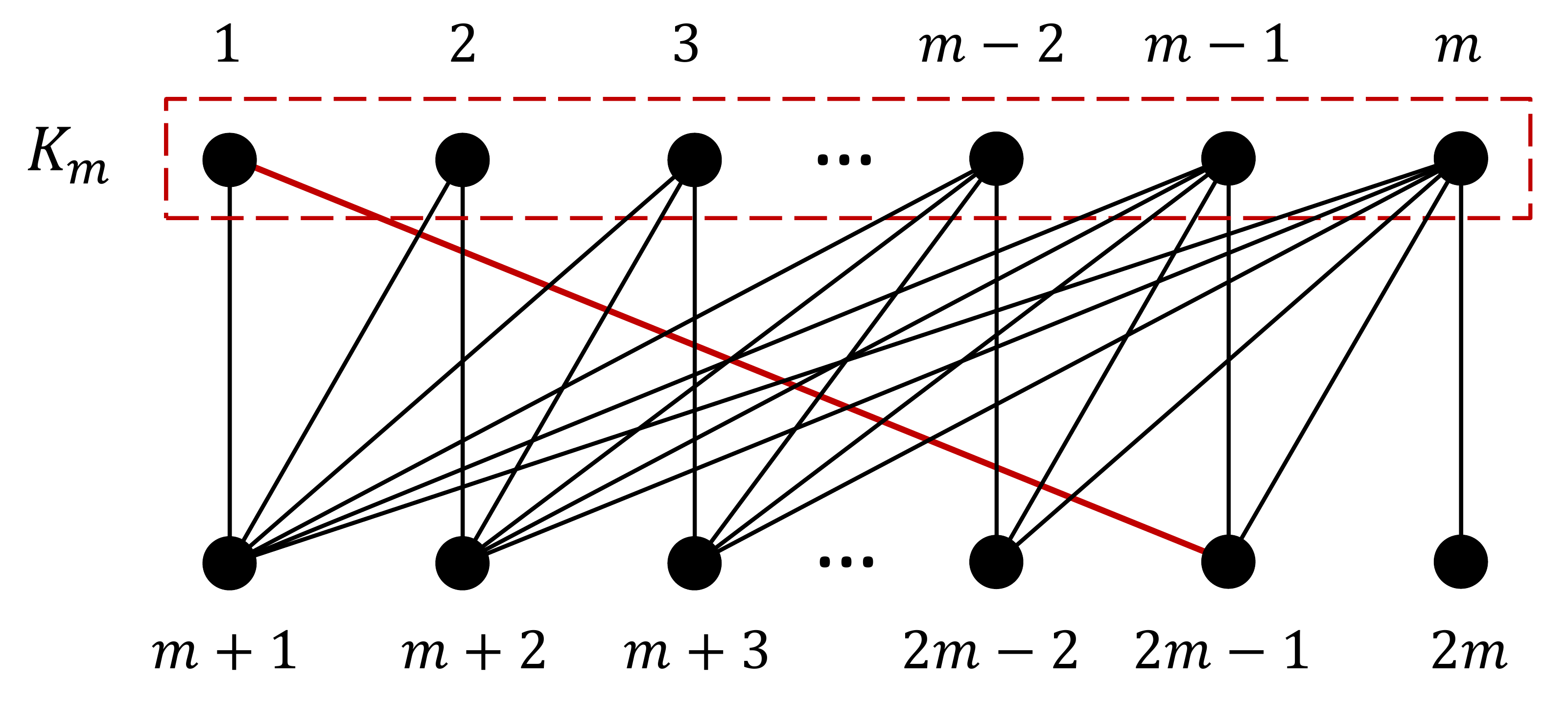}
	\caption{Graph $\mathcal{W}_{2m}$.}
	\label{fig4}
\end{figure}

For each vertex $v \in V(\mathcal{W}_{2m})$, let $Y_v$ denote its $(P_n)_r$-degree in the graph $\mathcal{W}_{2m}$.

The subsequent proposition will be instrumental in our analysis.

\begin{proposition} \label{p2}
	Let $1 \leq s \leq n$. In the graph $\mathcal{W}_{2m}$, for any subsets $U_1, U_2, \dots, U_s$ of $V(\mathcal{W}_{2m})$, the number of copies of $P_n$ containing a set of vertices $\{v_1, v_2, \dots, v_s\}$ such that $v_i \in U_i$ for each $i \in \{1, 2,\dots, s\}$ is at most
	\begin{equation} \label{eq7}    
		\prod_{i=1}^s |U_i| \cdot n! \binom{2m-s}{n-s}.
	\end{equation}
\end{proposition}

\begin{proof}
	To estimate the number of such subgraphs, first consider the choice of the set $\{v_1, v_2,\dots, v_s\}$. Since each $v_i \in U_i$, there are at most $\prod_{i=1}^s |U_i|$ ways to choose these $s$ vertices. The remaining $n-s$ vertices of the path $P_n$ can be selected from the $2m-s$ available vertices of $\mathcal{W}_{2m}$ in $\binom{2m-s}{n-s}$ ways. Finally, as $n$ vertices form the path $P_n$ in at most $n!$ ways, the product of these factors yields \eqref{eq7}.
\end{proof}

We now compare the $(P_n)_r$-degrees of vertices in $\mathcal{W}_{2m}$. For brevity, an $n$-path will denote any subgraph of $\mathcal{W}_{2m}$ isomorphic to $P_n$. We omit the trivial proofs of existence for the paths considered hereafter, as they follow directly from Remark~1.
\smallskip

\noindent \textbf{Remark 1.} The graph $\mathcal{W}_{2m}$ possesses the following structural properties:
\begin{itemize}
	\itemsep=0pt
	\item The vertices of $V_1$ (the upper level, see Fig. 4) form a clique.
	\item The neighborhood of vertex $2m-1$ is $N_{\mathcal{W}_{2m}}(2m-1)=\{1, m-1, m\}$.
	\item For any lower-level vertex $j \neq 2m-1$, an edge $(i, j) \in E(\mathcal{W}_{2m})$ exists if and only if $j-m \leq i \leq m$.
	\item For any $i_1 < i_2$ on the upper level and any vertex $j \notin \{i_2, 2m-1\}$, $(i_1, j) \in E(\mathcal{W}_{2m})$ implies $(i_2, j) \in E(\mathcal{W}_{2m})$.
	\item Similarly, for any $j_1 < j_2$ on the lower level with $j_2 \neq 2m-1$ and any vertex $i$, $(j_2, i) \in E(\mathcal{W}_{2m})$ implies $(j_1, i) \in E(\mathcal{W}_{2m})$.
\end{itemize}

\begin{lemma} \label{lemma9}
	If either $r \neq 1$, or $r = 1$ and the condition 
	\begin{equation} \label{eq8}
		\binom{m-3}{n-2} > 2n! \binom{2m-4}{n-4}
	\end{equation}
	holds, then $Y_{2m-1} > Y_{2m}$.
\end{lemma}

\begin{proof}    	
	Let $\mathcal{E}_8 = \{(1, 2m-1), (m-1, 2m-1)\}$ and $G_8 = \mathcal{W}_{2m} \setminus \mathcal{E}_8$. The vertices $2m-1$ and $2m$ are symmetric in $G_8$ as they share the same neighborhood $\{m\}$. Therefore, by \eqref{eq1}, we have
	\begin{align} \label{eq9}
		Y_{2m-1} - Y_{2m} &= |\mathcal{A}| - |\mathcal{B}|, \\
		\mathcal{A} &= \{F \in \mathcal{P}_n(\mathcal{W}_{2m}, 2m-1, r) \mid E(F) \cap \mathcal{E}_8 \neq \emptyset \}, \notag \\
		\mathcal{B} &= \{H \in \mathcal{P}_n(\mathcal{W}_{2m}, 2m, r) \mid E(H) \cap \mathcal{E}_8 \neq \emptyset \}. \notag
	\end{align}
	
We distinguish two cases depending on the condition of Lemma~\ref{lemma9}.
	
	\medskip
	\noindent \textbf{Case 1.} $r \neq 1$.
	
In this case, the set $\mathcal{B}$ is empty. Indeed, since $\deg_{\mathcal{W}_{2m}}(2m) = m$, vertex $2m$ must be an end-vertex of any $n$\nobreakdash-path, whereas every path in $\mathcal{B}$ requires it to correspond to the root $r$ with $1 < r \le \lceil n/2 \rceil$, a contradiction. In contrast, $\mathcal{A}$ is non-empty since it contains the $n$-path 
	\[
	(r-1, \, r-2, \dots, \, 1, \, 2m-1, \, m, \, m-1, \dots, \, m-n+r+1).
	\]
Thus, $|\mathcal{A}| > |\mathcal{B}|$, and \eqref{eq9} implies $Y_{2m-1} > Y_{2m}$.
	
\medskip
\noindent \textbf{Case 2.} $r = 1$ and \eqref{eq8} holds.
	
In this setting, we establish the following estimates for the cardinalities of $\mathcal{A}$ and $\mathcal{B}$:
\begin{align}                 			
	|\mathcal{A}| &\geq \binom{m-3}{n-2}, \label{eq10} \\
	|\mathcal{B}| &\leq 2n! \binom{2m-4}{n-4}. \label{eq11}
\end{align}    

To verify \eqref{eq10}, consider an arbitrary subset of $n-2$ vertices $u_1 < u_2 < \dots < u_{n-2}$ chosen from the set $\{2, 3, \dots, m-2\}$. Observe that the $n$\nobreakdash-path $(2m-1, 1, u_1, u_2, \dots, u_{n-2})$ belongs to $\mathcal{A}$. Since distinct subsets yield distinct $n$\nobreakdash-paths, $|\mathcal{A}|$ is at least the number of such choices, which is exactly $\binom{m-3}{n-2}$.
\smallskip	

The estimate \eqref{eq11} follows from Proposition~\ref{p2}, noting that every $n$\nobreakdash-path $H \in \mathcal{B}$ must contain the vertex set $\{2m, m, 2m-1, u\}$ for some $u \in \{1, m-1\}$. Indeed, since $r=1$, $H$ originates at $2m$ and must pass through its only neighbor $m$. Furthermore, the condition $E(H) \cap \mathcal{E}_8 \neq \emptyset$ implies that $H$ also contains $2m-1$ and at least one vertex from $\{1, m-1\}$.

Combining \eqref{eq8}--\eqref{eq11}, we obtain $Y_{2m-1} > Y_{2m}$.        
\end{proof}
	
\begin{lemma} \label{lemma10}
	If either $r = 1$ and 
	\begin{equation} \label{eq12}
		\frac{m}{4} \binom{m/4}{n-3} > 3n! \binom{2m-3}{n-3},
	\end{equation}
	or $r \neq 1$ and 
	\begin{equation} \label{eq13}
		\frac{m}{4} \binom{m/4}{n-4} > 5n! \binom{2m-4}{n-4},
	\end{equation}
	then $Y_{2m-2} > Y_{2m-1}$.
\end{lemma}
	
\begin{proof}	
	Let $\mathcal{E}_9 = \{(1, 2m-1), (m-2, 2m-2)\}$ and $G_9 = \mathcal{W}_{2m} \setminus \mathcal{E}_9$. The vertices $2m-2$ and $2m-1$ are symmetric in $G_9$ since they possess the same neighborhood $\{m-1, m\}$. Consequently, applying relation \eqref{eq1}, we deduce
	\begin{align} \label{eq14}
		Y_{2m-2} - Y_{2m-1} &= |\mathcal{A}| - |\mathcal{B}|, \\
		\mathcal{A} &= \{F \in \mathcal{P}_n(\mathcal{W}_{2m}, 2m-2, r) \mid E(F) \cap \mathcal{E}_9 \neq \emptyset \}, \notag \\
		\mathcal{B} &= \{H \in \mathcal{P}_n(\mathcal{W}_{2m}, 2m-1, r) \mid E(H) \cap \mathcal{E}_9 \neq \emptyset \}. \notag
	\end{align}
	
Since each of the conditions \eqref{eq12} and \eqref{eq13} implies that $m$ is divisible by 4, consider two vertex sets $U$ and $V$ in $\mathcal{W}_{2m}$, noting that $(u, v) \in E(\mathcal{W}_{2m})$ for all $u \in U$ and $v \in V$:
\begin{equation} \label{eq15}
	U = \left\{ \frac{m}{4}+1, \frac{m}{4}+2, \dots, \frac{m}{2} \right\}, \qquad V = \left\{ m+2, m+3, \dots, \frac{5m}{4}+1 \right\}.
\end{equation}
	
	Next, we partition the sets $\mathcal{A}$ and $\mathcal{B}$ as follows:		
	\begin{align} \label{eq16}
		\mathcal{A} &= \mathcal{A}_1 \sqcup \mathcal{A}_2 \sqcup \mathcal{A}_3, \quad \mathcal{B} = \mathcal{B}_1 \sqcup \mathcal{B}_2 \sqcup \mathcal{B}_3, \\
		\mathcal{A}_1 &= \{ F \in \mathcal{A} \mid (m-2, 2m-2) \in E(F), \ N_F(m-2) \cap V = \emptyset \}, \notag \\
		\mathcal{A}_2 &= \{ F \in \mathcal{A} \mid (m-2, 2m-2) \in E(F), \ N_F(m-2) \cap V \neq \emptyset \}, \notag \\
		\mathcal{A}_3 &= \{ F \in \mathcal{A} \mid (m-2, 2m-2) \notin E(F) \}, \notag \\
		\mathcal{B}_1 &= \{ H \in \mathcal{B} \mid (1, 2m-1) \in E(H), \ V(H) \cap \{m-2, 2m-2\} = \emptyset \}, \notag \\
		\mathcal{B}_2 &= \{ H \in \mathcal{B} \mid (1, 2m-1) \in E(H), \ V(H) \cap \{m-2, 2m-2\} \neq \emptyset \}, \notag \\
		\mathcal{B}_3 &= \{ H \in \mathcal{B} \mid (1, 2m-1) \notin E(H) \}. \notag
	\end{align}
		
Our immediate goal is to show that
\begin{equation} \label{eq17}			
	|\mathcal{A}_1| \geq |\mathcal{B}_1|.
\end{equation}
To this end, we construct an injective mapping $g_1 \colon \mathcal{B}_1 \to \mathcal{A}_1$. Let $H \in \mathcal{B}_1$. Since $(1, 2m-1) \in E(H)$, $V(H) \cap \{m-2, 2m-2\} = \emptyset$, and the vertex $2m-1$ serves as the root $r$ in $H$, the path can be represented as
\[
H = (x_1, x_2, \dots, x_a, 2m-1, 1, y_1, y_2, \dots, y_{n-a-2}),
\]
where $a \in \{r-1, n-r\}$, and all vertices $x_p$ and $y_q$ (if present in $H$) are distinct elements from $V(\mathcal{W}_{2m}) \setminus \{1, m-2, 2m-2, 2m-1\}$. The image of $H$ under $g_1$ is given by
\[
F=g_1(H) = (x_1, x_2, \dots, x_a, 2m-2, m-2, y_1,  y_2, \dots, y_{n-a-2}).
\]

To see that $g_1$ is well-defined, observe that $F \in \mathcal{A}$ and $(m-2, 2m-2) \in E(F)$. Moreover, the neighborhood $N_F(m-2)$ is either $\{2m-2\}$ or $\{y_1, 2m-2\}$. In the latter case, since $(1, y_1) \in E(H)$ and $y_1 \neq 2m-1$, it follows that $y_1 \in \{2, 3, \ldots, m+1\}$. Consequently, since $m \ge 6$, in both cases we have $N_F(m-2) \cap V = \emptyset$, thereby placing $F$ in $\mathcal{A}_1$. The injectivity of $g_1$ is clear, which concludes the proof of \eqref{eq17}.

\smallskip		
The analysis now splits into two scenarios, depending on which condition of Lemma~\ref{lemma10} is satisfied.
		
\medskip
\noindent \textbf{Case 1.} $r = 1$ and \eqref{eq12} holds. In this case, $m/4 \geq n-3$.
		
\smallskip
We derive the following estimates for the cardinalities of $\mathcal{A}_2$, $\mathcal{B}_2$, and $\mathcal{B}_3$:
\begin{align}                 			
	|\mathcal{A}_2| &\geq \frac{m}{4} \binom{m/4}{n-3}, \label{eq18} \\
	|\mathcal{B}_2| &\leq 2n! \binom{2m-3}{n-3}, \label{eq19} \\
	|\mathcal{B}_3| &\leq n! \binom{2m-3}{n-3}. \label{eq20}
\end{align}

To prove \eqref{eq18}, we choose $n-3$ vertices $u_1 < u_2 < \dots < u_{n-3}$ from $U$ and a vertex $v$ from $V$. Since $|U| = |V| = m/4$, there are exactly $\frac{m}{4} \binom{m/4}{n-3}$ ways to choose these vertices. With each choice, we injectively associate an $n$\nobreakdash-path $(2m-2, m-2, v, u_1, u_2, \dots, u_{n-3})$ belonging to $\mathcal{A}_2$, which establishes \eqref{eq18}.
\smallskip

Inequality \eqref{eq19} follows from Proposition~\ref{p2}, since every $n$-path in $\mathcal{B}_2$ contains the vertices $1$, $2m-1$, and $w_1$ for some $w_1 \in \{m-2, 2m-2\}$.
\smallskip

To find an upper bound for $|\mathcal{B}_3|$, recall that any $n$-path $H$ in $\mathcal{B}_3$ avoids the edge $(1, 2m-1)$, so its definition forces it to contain the edge $(m-2, 2m-2)$. In fact, $H$ contains the vertices $m-2$, $2m-2$, and $2m-1$. Thus, \eqref{eq20} holds by Proposition~\ref{p2}.
\smallskip
		
Combining \eqref{eq12} and \eqref{eq16}--\eqref{eq20}, we obtain
\begin{align*}		
	|\mathcal{A}| \geq |\mathcal{A}_1| + |\mathcal{A}_2| &\geq |\mathcal{B}_1| + \frac{m}{4} \binom{m/4}{n-3} \\
	&> |\mathcal{B}_1| + 3n! \binom{2m-3}{n-3} \geq |\mathcal{B}_1| + |\mathcal{B}_2| + |\mathcal{B}_3| = |\mathcal{B}|.
\end{align*}
Hence, $Y_{2m-2} > Y_{2m-1}$ in Case~1 by \eqref{eq14}.

\medskip		
\noindent \textbf{Case 2.} $r \ne 1$ and \eqref{eq13} holds. In this case, $m/4 \geq n-4$.

\smallskip
We establish the following estimates for the cardinalities of $\mathcal{A}_2$, $\mathcal{B}_2$, and $\mathcal{B}_3$:
\begin{align}                 					
	|\mathcal{A}_2| &\geq \frac{m}{4} \binom{m/4}{n-4}, \label{eq21} \\
	|\mathcal{B}_2| &\leq 4n! \binom{2m-4}{n-4}, \label{eq22} \\
	|\mathcal{B}_3| &\leq n! \binom{2m-4}{n-4}. \label{eq23}
\end{align}

To prove \eqref{eq21}, choose $n-4$ vertices $u_1 < u_2 < \dots < u_{n-4}$ from $U$ and a vertex $v \in V$. Since $|U| = |V| = m/4$, there are exactly $\frac{m}{4} \binom{m/4}{n-4}$ such choices. We map each selection to an $n$-path 
\[
(u_1, u_2, \dots, u_{r-2}, m, 2m-2, m-2, v, u_{r-1}, u_r, \dots, u_{n-4}) \in \mathcal{A}_2,
\]
which is well-defined because $0 \leq r-2 \leq \lceil n/2 \rceil - 2 \leq n-4$ for all $1 < r \leq \lceil n/2 \rceil$ and $n \geq 4$. The existence of this injective mapping confirms \eqref{eq21}.
\smallskip

Next, any $n$-path $H \in \mathcal{B}_2$ contains the edge $(1, 2m-1)$, with the vertex $2m-1$ as the root. Since $r \neq 1$ and $N_{\mathcal{W}_{2m}}(2m-1) = \{1, m-1, m\}$, this vertex must have a neighbor $w_2 \in \{m-1, m\}$ in $H$. Furthermore, $H$ intersects $\{m-2, 2m-2\}$ at some vertex $w_3$. Thus, the vertices $1, 2m-1, w_2, w_3$ belong to $H$, and \eqref{eq22} holds by Proposition~\ref{p2}.
\smallskip

To verify \eqref{eq23}, observe that any $n$-path $H \in \mathcal{B}_3$ contains the vertex $2m-1$, which corresponds to an internal root~$r$. Since $(1, 2m-1) \notin E(H)$, $(m-2, 2m-2) \in E(H)$, the structural restriction on $N_{\mathcal{W}_{2m}}(2m-1)$ forces $H$ to include  the edge $(m-1, 2m-1)$ and the vertex set $\{m-2, m-1, 2m-2, 2m-1\}$, immediately yielding \eqref{eq23} via Proposition~\ref{p2}.

It follows from \eqref{eq13}, \eqref{eq16}, \eqref{eq17}, and \eqref{eq21}--\eqref{eq23} that		
\begin{align*}		
	|\mathcal{A}| \geq |\mathcal{A}_1| + |\mathcal{A}_2| &\geq |\mathcal{B}_1| + \frac{m}{4} \binom{m/4}{n-4} \\
	&> |\mathcal{B}_1| + 5n! \binom{2m-4}{n-4} \geq |\mathcal{B}_1| + |\mathcal{B}_2| + |\mathcal{B}_3| = |\mathcal{B}|.
\end{align*}
In view of \eqref{eq14}, this gives the inequality $Y_{2m-2} > Y_{2m-1}$ for Case~2.
\end{proof}	
	
\begin{lemma} \label{lemma11}
	If 
	\begin{equation} \label{eq24}		
		\binom{m-4}{n-3} > 9n! \binom{2m-4}{n-4},
	\end{equation}	
	then 
	\[
	Y_{i} > Y_{i+1} \quad \text{for each } i \in \{m+1, m+2, \dots, 2m-3\}.
	\]		
\end{lemma}

\begin{proof}	
Fix $i \in \{m+1, m+2, \dots, 2m-3\}$. Let $\mathcal{E}_{10} = \{(i-m, i)\}$ and $G_{10} = \mathcal{W}_{2m} \setminus \mathcal{E}_{10}$. The vertices $i$ and $i+1$ are symmetric in $G_{10}$ since they have identical neighborhoods:
\[
N_{G_{10}}(i) = N_{G_{10}}(i+1) = \{i-m+1, i-m+2, \dots , m\}.
\]
Consequently, by equation~\eqref{eq1}, we obtain
\begin{align} \label{eq25}
	Y_{i} - Y_{i+1} &= |\mathcal{A}| - |\mathcal{B}|, \\
	\mathcal{A} &= \{F \in \mathcal{P}_n(\mathcal{W}_{2m}, i, r) \mid (i-m, i) \in E(F) \}, \notag \\
	\mathcal{B} &= \{H \in \mathcal{P}_n(\mathcal{W}_{2m}, i+1, r) \mid (i-m, i) \in E(H) \}. \notag
\end{align}
						
We note that since every $H \in \mathcal{B}$ contains the vertex $i+1$ and the edge $(i-m, i)$, the path $H$ must belong to one of two types:
\begin{align*}
	&\text{Type I:} \quad (\dots, i+1, \dots, i, i-m, \dots), \\
	&\text{Type II:} \quad (\dots, i+1, \dots, i-m, i, \dots).
\end{align*}

Let $M = \{m-2, m-1, m\}$. We decompose $\mathcal{A}$ and $\mathcal{B}$ as follows:
\begin{align} \label{eq26}
	\mathcal{A} &= \mathcal{A}_1 \sqcup \mathcal{A}_2 \sqcup \mathcal{A}_3 \sqcup \mathcal{A}_4, \quad \mathcal{B} = \mathcal{B}_1 \sqcup \mathcal{B}_2 \sqcup \mathcal{B}_3, \\
	\mathcal{A}_1 &= \{F \in \mathcal{A} \mid V(F) \cap M = \{m\}\}, \notag \\
	\mathcal{A}_2 &= \{F \in \mathcal{A} \mid V(F) \cap M = \{m-1\}\}, \notag \\
	\mathcal{A}_3 &= \{F \in \mathcal{A} \mid V(F) \cap M = \{m-2\}\}, \notag \\
	\mathcal{A}_4 &= \mathcal{A} \setminus (\mathcal{A}_1 \cup \mathcal{A}_2 \cup \mathcal{A}_3), \notag \\
	\mathcal{B}_1 &= \{H \in \mathcal{B} \mid H \text{ is of Type I}, \ V(H) \cap M = \emptyset \}, \notag \\
	\mathcal{B}_2 &= \{H \in \mathcal{B} \mid H \text{ is of Type II}, \ V(H) \cap M = \emptyset \}, \notag \\
	\mathcal{B}_3 &= \{H \in \mathcal{B} \mid V(H) \cap M \neq \emptyset \}. \notag
\end{align}
		
Next, we prove the estimates:
	\begin{align}				
		|\mathcal{A}_1| &\geq |\mathcal{B}_1|, \label{eq27} \\
		|\mathcal{A}_2| &\geq |\mathcal{B}_2|, \label{eq28} \\
		|\mathcal{A}_3| &\geq \binom{m-4}{n-3}, \label{eq29} \\
		|\mathcal{B}_3| &\leq 3n! \binom{2m-4}{n-4}. \label{eq30}
	\end{align}	
		
Inequality \eqref{eq27} holds via an injective mapping $g_2 \colon \mathcal{B}_1 \to \mathcal{A}_1$. By definition, every path $H \in \mathcal{B}_1$ can be written as
\[
H = (x_1, x_2, \ldots, x_a, i+1, y_1, y_2, \ldots, y_b, i, i-m, z_1, z_2, \ldots, z_{n-a-b-3}),
\]
where $a \in \{r-1, n-r\}$, and all vertices $x_p$, $y_q$, and $z_j$ actually occurring in $H$ are distinct and belong to $V(\mathcal{W}_{2m}) \setminus (\{i, i+1, i-m\} \cup M)$. The image of $H$ under $g_2$ is given by 
\[
g_2(H) = (x_1, x_2, \ldots, x_a, i, i-m, z_1,  z_2, \ldots, z_{n-a-b-3}, m, y_1, y_2, \ldots, y_b).
\]

The estimate \eqref{eq28} is obtained analogously via an injective mapping $g_3 \colon \mathcal{B}_2 \to \mathcal{A}_2$. Every path $H \in \mathcal{B}_2$ can be represented as
\[
H = (x_1, x_2, \dots, x_a, i+1, y_1, y_2, \dots, y_b, i-m, i, z_1, z_2, \dots, z_{n-a-b-3}),
\]
where the constraints on $a$ and the remaining vertices are identical to those in the proof of \eqref{eq27}. Under $g_3$, this path is mapped to
\[
g_3(H) = (x_1, x_2, \dots, x_a, i, i-m, y_b, y_{b-1}, \dots, y_1, m-1, z_1, z_2, \dots, z_{n-a-b-3}). 
\]

To establish \eqref{eq29}, we associate each subset of $n-3$ vertices $u_1 < u_2 < \dots < u_{n-3}$ from $\{1, 2, \dots, m-3\} \setminus \{i-m\}$ with an $n$-path in $\mathcal{A}_3$ defined by
\[ 
\begin{cases} 
	(i, i-m, u_1, u_2, \dots, u_{n-3}, m-2), & \text{if } r=1, \\ 
	(u_1, u_2, \dots, u_{r-2}, i-m, i, m-2, u_{r-1}, u_r, \dots, u_{n-3}), & \text{if } r \neq 1. 
\end{cases} 
\] 
Since this construction is injective and the number of such vertex subsets is exactly $\binom{m-4}{n-3}$, the inequality \eqref{eq29} is immediate.

Finally, \eqref{eq30} follows from Proposition~\ref{p2} and the fact that every $n$-path in $\mathcal{B}_3$ contains the vertex set $\{i, i-m, i+1, u\}$ for some $u \in M$.

Combining \eqref{eq24} and \eqref{eq26}--\eqref{eq30} yields the chain of inequalities
\begin{align*}		
	|\mathcal{A}| \geq |\mathcal{A}_1| + |\mathcal{A}_2|+|\mathcal{A}_3| &\geq |\mathcal{B}_1| +|\mathcal{B}_2|+ \binom{m-4}{n-3} \\
	&> |\mathcal{B}_1| +|\mathcal{B}_2|+ 9n! \binom{2m-4}{n-4} > |\mathcal{B}_1| + |\mathcal{B}_2| + |\mathcal{B}_3| = |\mathcal{B}|.
\end{align*}
Consequently, by~\eqref{eq25}, $Y_{i} > Y_{i+1}$ holds for each $i \in \{m+1, m+2, \dots, 2m-3\}$.
	\end{proof}	
		
\begin{lemma} \label{lemma12}
	If \eqref{eq24} is valid, then  $Y_{1} > Y_{m+1}$.		
\end{lemma}

\begin{proof}	
	Let $\mathcal{E}_{11} = \{(1, 2m-1)\}$ and $G_{11} = \mathcal{W}_{2m} \setminus \mathcal{E}_{11}$. The vertices $1$ and $m+1$ are symmetric in $G_{11}$ as their closed neighborhoods coincide:
	\[
	N_{G_{11}}[1] = N_{G_{11}}[m+1] = \{1, 2, \dots , m+1\}.
	\]
	By virtue of \eqref{eq1}, we have	
	\begin{align} \label{eq31}
		Y_{1} - Y_{m+1} &= |\mathcal{A}| - |\mathcal{B}|, \\
		\mathcal{A} &= \{F \in \mathcal{P}_n(\mathcal{W}_{2m}, 1, r) \mid (1, 2m-1) \in E(F) \}, \notag \\
		\mathcal{B} &= \{H \in \mathcal{P}_n(\mathcal{W}_{2m}, m+1, r) \mid (1, 2m-1) \in E(H) \}. \notag
	\end{align}
	
We introduce the following partitions of $\mathcal{A}$ and $\mathcal{B}$:
\begin{align} \label{eq32}
	\mathcal{A} &= \mathcal{A}_1 \sqcup \mathcal{A}_2 \sqcup \mathcal{A}_3, \quad \mathcal{B} = \mathcal{B}_1 \sqcup \mathcal{B}_2, \\
	\mathcal{A}_1 &= \{F \in \mathcal{A} \mid V(F) \cap \{m-1, m\} = \{m\}\}, \notag \\
	\mathcal{A}_2 &= \{F \in \mathcal{A} \mid V(F) \cap \{m-1, m\} = \{m-1\}\}, \notag \\
	\mathcal{A}_3 &= \mathcal{A} \setminus (\mathcal{A}_1 \cup \mathcal{A}_2), \notag \\
	\mathcal{B}_1 &= \{H \in \mathcal{B} \mid V(H) \cap \{m-1, m\} = \emptyset \}, \notag \\
	\mathcal{B}_2 &= \{H \in \mathcal{B} \mid V(H) \cap \{m-1, m\} \neq \emptyset \}. \notag
\end{align}

The cardinalities of $\mathcal{A}_1$, $\mathcal{B}_1$, $\mathcal{A}_2$, and $\mathcal{B}_2$ satisfy the inequalities:
\begin{align}                 				
	|\mathcal{A}_1| &\geq |\mathcal{B}_1|, \label{eq33} \\
	|\mathcal{A}_2| &\geq \binom{m-4}{n-3}, \label{eq34} \\
	|\mathcal{B}_2| &\leq 2n! \binom{2m-4}{n-4}. \label{eq35}
\end{align}
		
To prove \eqref{eq33}, we construct an injective mapping $g_4 \colon \mathcal{B}_1 \to \mathcal{A}_1$. Let $H \in \mathcal{B}_1$. Since $(1, 2m-1) \in E(H)$, $N_{\mathcal{W}_{2m}}(2m-1) = \{1, m-1, m\}$, and $V(H) \cap \{m-1, m\} = \emptyset$, the path $H$ must terminate at $2m-1$. Hence, by the structure of $\mathcal{B}_1$, it is written as
\[
H = (x_1, x_2, \dots, x_a, m+1, y_1, y_2, \dots, y_{n-a-3}, 1, 2m-1),
\]
where $a \in \{r-1, n-r\}$, and all vertices $x_p$ and $y_q$, if they exist in $H$, are distinct and belong to $V(\mathcal{W}_{2m}) \setminus \{1, m-1, m, m+1, 2m-1\}$. The image of $H$ under $g_4$ is given by
\[
g_4(H) = (x_1, x_2, \dots, x_a, 1, 2m-1, m, y_1, y_2, \dots, y_{n-a-3}).
\]

To establish \eqref{eq34}, consider $n-3$ arbitrary vertices $u_1 < u_2 < \dots < u_{n-3}$ chosen from $\{2, 3, \dots, m-3\}$. We injectively map each such selection to the $n$-path in $\mathcal{A}_2$ defined by
\[
(u_1, u_2, \dots, u_{r-1}, 1, 2m-1, m-1, u_{r}, u_{r+1}, \dots, u_{n-3}).
\]
Since there are exactly $\binom{m-4}{n-3}$ such choices, the desired result holds.

\smallskip		
The upper bound \eqref{eq35} is a direct consequence of Proposition~\ref{p2}, since the vertex set of any $n$-path in $\mathcal{B}_2$ inevitably contains $\{1, m+1, 2m-1, u\}$ for some $u \in \{m-1, m\}$.

\smallskip		
Estimates \eqref{eq24} and \eqref{eq32}--\eqref{eq35} lead to 
\begin{align*}		
	|\mathcal{A}| \geq |\mathcal{A}_1| + |\mathcal{A}_2| &\geq |\mathcal{B}_1| + \binom{m-4}{n-3} \\
	&> |\mathcal{B}_1| + 9n! \binom{2m-4}{n-4} > |\mathcal{B}_1| + |\mathcal{B}_2| = |\mathcal{B}|,
\end{align*}
which, by \eqref{eq31}, implies $Y_1 > Y_{m+1}$.
\end{proof}
	
\begin{lemma} \label{lemma13}
		If 
		\begin{equation} \label{eq36}
			\binom{m-3}{n-2} > 5n! \binom{2m-3}{n-3},
		\end{equation}
		then $Y_2 > Y_1$.
	\end{lemma}
		
\begin{proof}	
Let $\mathcal{E}_{12} = \{(1, 2m-1), (2, m+2)\}$ and $G_{12} = \mathcal{W}_{2m} \setminus \mathcal{E}_{12}$. The vertices $1$ and $2$ are symmetric in $G_{12}$ since their closed neighborhoods are identical:
\[
N_{G_{12}}[1] = N_{G_{12}}[2] = \{1, 2, \dots, m+1\}.
\]
Therefore, by \eqref{eq1}, the difference between the $(P_n)_r$-degrees of vertices $2$ and $1$ in $\mathcal{W}_{2m}$ is given by
\begin{align} \label{eq37}
	Y_{2} - Y_{1} &= |\mathcal{A}| - |\mathcal{B}|, \\
	\mathcal{A} &= \{F \in \mathcal{P}_n(\mathcal{W}_{2m}, 2, r) \mid E(F) \cap \mathcal{E}_{12}\neq \emptyset \}, \notag \\
	\mathcal{B} &= \{H \in \mathcal{P}_n(\mathcal{W}_{2m}, 1, r) \mid E(H) \cap \mathcal{E}_{12} \neq \emptyset \}. \notag
\end{align}

The proof now proceeds by cases depending on whether $r = 2$ or $r \neq 2$.
\smallskip
		
\noindent \textbf{Case 1.} $r = 2$.

Let $L=\{2, m-1, m, m+2\}$. Here, the sets $\mathcal{A}$ and $\mathcal{B}$ are partitioned as follows:
\begin{align} \label{eq38}
	\mathcal{A} &= \mathcal{A}_1 \sqcup \mathcal{A}_2 \sqcup \mathcal{A}_3, \quad \mathcal{B} = \mathcal{B}_1 \sqcup \mathcal{B}_2 \sqcup \mathcal{B}_3, \\
	\mathcal{A}_1 &= \{F \in \mathcal{A} \mid (2, m+2) \in E(F), \ \deg_F(m+2) = 1\}, \notag \\
	\mathcal{A}_2 &= \{F \in \mathcal{A} \mid (2, m+2) \in E(F), \ \deg_F(m+2) \neq 1 \}, \notag \\
	\mathcal{A}_3 &= \{F \in \mathcal{A} \mid (2, m+2) \notin E(F)\}, \notag \\
	\mathcal{B}_1 &= \{H \in \mathcal{B} \mid (1, 2m-1) \in E(H), \ V(H) \cap L = \emptyset\}, \notag \\
	\mathcal{B}_2 &= \{H \in \mathcal{B} \mid (1, 2m-1) \in E(H), \ V(H) \cap L \neq \emptyset \}, \notag \\
	\mathcal{B}_3 &= \{H \in \mathcal{B} \mid (1, 2m-1) \notin E(H)\}. \notag
\end{align}	

Our next goal is to establish the estimates for the cardinalities of these subsets:
\begin{align}			
	|\mathcal{A}_1| &\geq |\mathcal{B}_1|, \label{eq39} \\
	|\mathcal{A}_2| &\geq \binom{m-3}{n-2}, \label{eq40} \\
	|\mathcal{B}_2| &\leq 4n! \binom{2m-3}{n-3}, \label{eq41} \\
	|\mathcal{B}_3| &\leq n! \binom{2m-3}{n-3}. \label{eq42} 
\end{align}

Inequality \eqref{eq39} is ensured by the existence of an injection $g_5 \colon \mathcal{B}_1 \to \mathcal{A}_1$. Let $H \in \mathcal{B}_1$. The vertex $2m-1$ must be an end-vertex of $H$, since $(1, 2m-1) \in E(H)$, $V(H) \cap L = \emptyset$, and $N_{\mathcal{W}_{2m}}(2m-1) = \{1, m-1, m\}$. Thus, the path $H$ takes the form
\[
H = (2m-1, 1, x_1, x_2, \dots, x_{n-2}),
\]
where all vertices $x_p$ are distinct and belong to $V(\mathcal{W}_{2m}) \setminus (L \cup \{1, 2m-1\})$. The image of $H$ under $g_5$ is given by
\[
g_5(H) = (m+2, 2, x_1, x_2, \dots, x_{n-2}).
\]

To obtain \eqref{eq40}, observe that there are exactly $\binom{m-3}{n-2}$ collections of $n-2$ vertices $u_1 < u_2 < \dots < u_{n-2}$ chosen from the set $\{3, 4, \dots, m-1\}$. Since $n \ge 4$, we can injectively map each such collection to an $n$-path in $\mathcal{A}_2$ of the form $(u_1, 2, m+2, u_2, u_3, \dots, u_{n-2})$, confirming the desired inequality.
\smallskip

To prove \eqref{eq41}, we use Proposition~\ref{p2}, noting that every path $H \in \mathcal{B}_2$ contains the vertices $1$, $2m-1$, and some $v \in L$.
\smallskip

To establish \eqref{eq42}, recall that every path $H \in \mathcal{B}_3$ contains the vertex $1$ and the edge $(2, m+2)$, so that $\{1, 2, m+2\} \subset V(H)$ and the assertion follows from Proposition~\ref{p2}.
\smallskip

The estimates \eqref{eq36} and \eqref{eq38}--\eqref{eq42} imply $|\mathcal{A}| > |\mathcal{B}|$:
\begin{align*}				
	|\mathcal{A}| \geq |\mathcal{A}_1| + |\mathcal{A}_2| &\geq |\mathcal{B}_1| + \binom{m-3}{n-2} \\
	&> |\mathcal{B}_1| + 5n! \binom{2m-3}{n-3} \geq |\mathcal{B}_1| + |\mathcal{B}_2| + |\mathcal{B}_3| = |\mathcal{B}|.
\end{align*}
Consequently, by \eqref{eq37}, this strict inequality guarantees that $Y_2 > Y_1$ for $r=2$.
\smallskip
	
\noindent \textbf{Case 2.} $r \ne 2$.
	
In this case, we have
\begin{align}				
	|\mathcal{A}| &\geq \binom{m-3}{n-2}, \label{eq43} \\
	|\mathcal{B}| &\leq 3n! \binom{2m-3}{n-3}. \label{eq44}	
\end{align}
	
The lower bound \eqref{eq43} is established by assigning each of the $\binom{m-3}{n-2}$ collections of $n-2$ vertices $u_1 < u_2 < \dots < u_{n-2}$ from $\{3, 4, \dots, m-1\}$ to an $n$-path in $\mathcal{A}$ of the form
\[
(u_1, u_2, \dots, u_{r-1}, 2, m+2, u_r, u_{r+1}, \dots, u_{n-2}).
\]
The injectivity of this mapping confirms the assertion.
\smallskip

To evaluate $|\mathcal{B}|$, we recall that every path $H \in \mathcal{B}$  contains the vertex $1$ and necessarily intersects the edge set $\mathcal{E}_{12} = \{(1, 2m-1), (2, m+2)\}$. If $H$ contains the edge $(2, m+2)$, then $\{1, 2, m+2\} \subset V(H)$, and Proposition~\ref{p2} restricts the number of such paths to $n! \binom{2m-3}{n-3}$. If $H$ contains the edge $(1, 2m-1)$, the conditions $N_{\mathcal{W}_{2m}}(2m-1) = \{1, m-1, m\}$ and $r \neq 2$ ensure that $H$ must pass through some vertex $v \in \{m-1, m\}$, implying that $\{1, 2m-1, v\} \subset V(H)$. By Proposition~\ref{p2}, there are at most $2n! \binom{2m-3}{n-3}$ such paths.  Summing these estimates validates \eqref{eq44}.
\smallskip

The inequalities \eqref{eq36}, \eqref{eq43}, and \eqref{eq44} imply that $|\mathcal{A}| > |\mathcal{B}|$, which, by \eqref{eq37}, yields $Y_2 > Y_1$ for $r \neq 2$.
\end{proof}		
			
\begin{lemma} \label{lemma14}
	If \eqref{eq24} holds, then
	\[
	Y_{i+1} > Y_{i} \quad \text{for each } i \in \{2, 3, \dots, m-4\}.
	\]	
\end{lemma}
	
\begin{proof}	
	Fix $i \in \{2, 3, \dots, m-4\}$. Let $\mathcal{E}_{13} = \{(i+1, m+i+1)\}$ and $G_{13} = \mathcal W_{2m} \setminus \mathcal{E}_{13}$. Having identical closed neighborhoods, the vertices $i$ and $i+1$ are symmetric in $G_{13}$:
	\[
	N_{G_{13}}[i] = N_{G_{13}}[i+1] = \{1, 2, \dots , m+i\}.
	\]

From here on, the proof of Lemma~\ref{lemma14} follows along the same lines as that of Lemma~\ref{lemma11}, with the triple of vertices $\langle i, i+1, i-m \rangle$ replaced by $\langle i+1, i, m+i+1 \rangle$. The only difference lies in the justification of bound \eqref{eq29}, which is provided below.

\smallskip
		
Consider the $\binom{m-4}{n-3}$ collections of vertices $u_1 < u_2 < \dots < u_{n-3}$ chosen from the set $\{1, 2, \dots, m-3\} \setminus \{i+1\}$. Each such vertex selection  maps injectively to the $n$-path in~$\mathcal{A}_3$ given by
\[
(u_1, u_2, \dots, u_{r-1}, i+1, m+i+1, m-2, u_r, u_{r+1}, \dots, u_{n-3}).
\]
This confirms the validity of estimate \eqref{eq29}.
\end{proof}	
		
\begin{lemma} \label{lemma15}
	If \eqref{eq24} is satisfied, then 
	$Y_{m-2} > Y_{m-3}$.		
\end{lemma}
\begin{proof}
	Let $\mathcal{E}_{14} = \{(m-2, 2m-2)\}$ and $G_{14} = \mathcal{W}_{2m} \setminus \mathcal{E}_{14}$. The vertices $m-3$ and $m-2$ are symmetric in $G_{14}$, as their closed neighborhoods coincide:
	\[
	N_{G_{14}}[m-3] = N_{G_{14}}[m-2] = \{1, 2, \dots , 2m-3\}.
	\]
Consequently, in view of \eqref{eq1}, we have 
	\begin{align} \label{eq45}
		Y_{m-2} - Y_{m-3} &= |\mathcal{A}| - |\mathcal{B}|, \\
		\mathcal{A} &= \{F \in \mathcal{P}_n(\mathcal{W}_{2m}, m-2, r) \mid (m-2, 2m-2) \in E(F) \}, \notag \\
		\mathcal{B} &= \{H \in \mathcal{P}_n(\mathcal{W}_{2m}, m-3, r) \mid (m-2, 2m-2) \in E(H) \}. \notag
	\end{align}
	
We decompose each of the sets $\mathcal{A}$ and $\mathcal{B}$ into disjoint classes:
	\begin{align} \label{eq46}
		\mathcal{A} &= \mathcal{A}_1 \sqcup \mathcal{A}_2, \quad \mathcal{B} = \mathcal{B}_1 \sqcup \mathcal{B}_2, \\
		\mathcal{A}_1 &= \{F \in \mathcal{A} \mid m \in V(F)\}, \notag \\
		\mathcal{A}_2 &= \{F \in \mathcal{A} \mid m \notin V(F)\}, \notag \\
		\mathcal{B}_1 &= \{H \in \mathcal{B} \mid V(H) \cap \{m-1, m\} = \emptyset\}, \notag \\
		\mathcal{B}_2 &= \{H \in \mathcal{B} \mid V(H) \cap \{m-1, m\} \neq \emptyset \}. \notag 	
	\end{align}
	
The following inequalities will be proved next:
	\begin{align}		
		|\mathcal{A}_1| &\geq |\mathcal{B}_1|, \label{eq47} \\
		|\mathcal{A}_2| &\geq \binom{m-4}{n-3} , \label{eq48} \\
		|\mathcal{B}_2| &\leq 2n! \binom{2m-4}{n-4}. \label{eq49}
	\end{align}
	
To establish \eqref{eq47}, we construct an injection $g_6 \colon \mathcal{B}_1 \to \mathcal{A}_1$. Note that every $H \in \mathcal{B}_1$ contains the edge $(m-2, 2m-2)$. Since $N_{\mathcal{W}_{2m}}(2m-2) = \{m-2, m-1, m\}$, the condition $V(H) \cap \{m-1, m\} = \emptyset$ forces $2m-2$ to be an end-vertex of $H$. Hence, 
\[
H = (x_1, x_2, \dots, x_a, m-3, y_1, y_2, \dots, y_{n-a-3}, m-2, 2m-2),
\]
where $a \in \{r-1, n-r\}$, and the remaining vertices $x_p$ and $y_q$  present in $H$ are distinct and belong to the set $V(\mathcal{W}_{2m}) \setminus \{m-3, m-2, m-1, m, 2m-2\}$. For each such $H$, its image is given by the $n$-path 
\[
g_6(H) = (x_1, x_2, \dots, x_a, m-2, 2m-2, m, y_1, y_2, \dots, y_{n-a-3}).
\]

The lower bound \eqref{eq48} is obtained by defining an injection from the set of $\binom{m-4}{n-3}$ vertex sequences $u_1 < u_2 < \dots < u_{n-3}$ in $\{2, 3, \dots, m-3\}$ into $\mathcal{A}_2$. This mapping assigns to each such sequence the $n$-path 
\[ 
(u_1, u_2, \dots, u_{r-1}, m-2, 2m-2, m-1, u_r, u_{r+1}, \dots, u_{n-3}) .  
\] 

The estimate \eqref{eq49} follows from Proposition~\ref{p2} and the fact that each $H \in \mathcal{B}_2$ contains the vertices $m-3$, $m-2$, $2m-2$, and some $v \in \{m-1, m\}$.
\smallskip
	
By \eqref{eq24} and \eqref{eq46}--\eqref{eq49}, the following chain leads to the strict inequality $|\mathcal{A}| > |\mathcal{B}|$:
\begin{align*}			
	|\mathcal{A}| = |\mathcal{A}_1| + |\mathcal{A}_2| &\geq |\mathcal{B}_1| + \binom{m-4}{n-3} \\
	&> |\mathcal{B}_1| + 9n! \binom{2m-4}{n-4} > |\mathcal{B}_1| + |\mathcal{B}_2| = |\mathcal{B}|.
\end{align*}
Together with \eqref{eq45}, this directly implies that $Y_{m-2} > Y_{m-3}$.
\end{proof}
		
\begin{lemma} \label{lemma16}
	If \eqref{eq24} holds, then $Y_{m-1} > Y_{m-2}$.
\end{lemma}

\begin{proof}
Due to the structure of the graph~$\mathcal{W}_{2m}$, the proof is analogous to that of Lemma~\ref{lemma15}, where the quadruple $\langle m-3, m-2, m-1, 2m-2 \rangle$ is replaced by $\langle m-2, m-1, 1, 2m-1 \rangle$.
\end{proof}
	
Next, we consider the vertex $2m$, whose position in the overall $(P_n)_r$-degree ordering depends crucially on the parameters $n$ and $r$.
		
\begin{lemma} \label{lemma17} 
	The following statements hold:
	\begin{enumerate}
		\item If $r = 2$ and \eqref{eq36} is satisfied, then $Y_{m} > Y_{m-1}$;
		\item If $r \neq 2$, then $Y_{m-1} > Y_{m}$.
	\end{enumerate}
\end{lemma}

\begin{proof}
	The vertices $m-1$ and $m$ are symmetric in the graph $G_{15}=\mathcal{W}_{2m} \setminus \{(m, 2m)\}$ since $N_{G_{15}}[m-1] = N_{G_{15}}[m] = \{1, 2, \dots, 2m-1\}$. Hence, by \eqref{eq1}, we have
	\begin{align} \label{eq50}
		Y_{m-1} - Y_{m} &= |\mathcal{A}| - |\mathcal{B}|, \\
		\mathcal{A} &= \{F \in \mathcal{P}_n(\mathcal{W}_{2m}, m-1, r) \mid (m, 2m) \in E(F)\}, \notag \\
		\mathcal{B} &= \{H \in \mathcal{P}_n(\mathcal{W}_{2m}, m, r) \mid (m, 2m) \in E(H)\}. \notag
	\end{align}
		
We distinguish two cases by the conditions of Lemma~\ref{lemma17}.	
\smallskip
		
\noindent \textbf{Case 1:} $r=2$ and \eqref{eq36} is valid.

To estimate $|\mathcal{B}|$ from below, with each choice of vertices $u_1 < u_2 < \dots < u_{n-2}$ within the set $\{1, 2, \dots, m-3\}$, we may associate the $n$-path $(2m, m, u_1, u_2, \dots, u_{n-2}) \in \mathcal{B}$. Since there are exactly $\binom{m-3}{n-2}$ such choices, the injectivity of this assignment yields 
\begin{equation} \label{eq51}						
	|\mathcal{B}| \geq \binom{m-3}{n-2}.
\end{equation}

Since every $F \in \mathcal{A}$ contains the vertices $m-1$, $m$, and $2m$, Proposition~\ref{p2} gives the upper bound
\begin{equation} \label{eq52}				
	|\mathcal{A}| \leq n! \binom{2m-3}{n-3}.
\end{equation}
Combining \eqref{eq36}, \eqref{eq51}, and \eqref{eq52}, we arrive at the strict inequality $|\mathcal{B}| > |\mathcal{A}|$. In view of~\eqref{eq50}, this directly confirms that $Y_{m} > Y_{m-1}$.
\smallskip

\noindent \textbf{Case 2:} $r \neq 2$.

By definition, every path in $\mathcal{B}$ contains the edge $(m, 2m)$. Since $N_{\mathcal{W}_{2m}}(2m) = \{m\}$, any such path must have $2m$ as an end-vertex, forcing $m$ to be the second vertex from one of its ends. This directly contradicts the condition $r \neq 2$, implying that $|\mathcal{B}| = 0$.
\smallskip

Conversely, the set $\mathcal{A}$ is non-empty as it contains the $n$-path
\[
(1, 2, \dots, r-1, m-1, r, r+1, \dots, n-3, m, 2m).
\]
Thus, we obtain $|\mathcal{A}| > |\mathcal{B}|$, which together with \eqref{eq50} leads to $Y_{m-1} > Y_{m}$.
\end{proof}
	
\begin{lemma} \label{lemma18} 
	The following assertions hold:
	\begin{enumerate}
		\item If $\langle n, r \rangle = \langle 5, 3 \rangle$, then $Y_{m} > Y_{m-2}$;
		\item If $r \neq 2$, $\langle n, r \rangle \neq \langle 5, 3 \rangle$, and \eqref{eq13} is satisfied, then $Y_{m-2} > Y_{m}$.
	\end{enumerate}
\end{lemma}	
\begin{proof}	
The vertices $m-2$ and $m$ are symmetric in $G_{16} = \mathcal{W}_{2m} \setminus \{(m, 2m-1), (m, 2m)\}$, as $N_{G_{16}}[m-2] = N_{G_{16}}[m] = \{1, 2, \dots, 2m-2\}$. Consequently, \eqref{eq1} yields
\begin{align} \label{eq53}
	Y_{m-2} - Y_{m} &= |\mathcal{A}| - |\mathcal{B}|, \\
	\mathcal{A} &= \{F \in \mathcal{P}_n(\mathcal{W}_{2m}, m-2, r) \mid E(F) \cap \{(m, 2m-1), (m, 2m)\} \neq \emptyset \}, \notag \\
	\mathcal{B} &= \{H \in \mathcal{P}_n(\mathcal{W}_{2m}, m, r) \mid E(H) \cap \{(m, 2m-1), (m, 2m)\} \neq \emptyset \}. \notag
\end{align} 

Further, two cases are examined in accordance with the assertions of Lemma~\ref{lemma18}.
\bigskip	
	
\noindent \textbf{Case 1.} $\langle n, r \rangle = \langle 5, 3 \rangle$.

The set $\mathcal{A}$ admits a partition into four disjoint classes:
\begin{flalign*}
	&\qquad \mathcal{A}_1 = \{ (x, y, m-2, m, z) \mid x \in \{1, 2, \dots, m-3\}, \\
	&\qquad \qquad \quad y \in \{1, 2, \dots, m+x\} \setminus \{x, m-2, m\}, \, z \in \{2m, 2m-1\} \}, \\[1ex]	
	&\qquad \qquad \qquad \qquad \qquad \quad |\mathcal{A}_1| = 2\sum_{x=1}^{m-3} (m+x-3) = (m-3)(3m-8); \\[1.5ex]
	&\qquad \mathcal{A}_2 = \{ (m-1, y, m-2, m, z) \mid y \in \{1, 2, \dots, 2m-2\} \setminus \{m-2, m-1, m\}, \\
	&\qquad \qquad \quad z \in \{2m, 2m-1\} \}, \\[1ex]
	&\qquad \qquad \qquad \qquad \qquad \quad |\mathcal{A}_2| = 2(2m-5); \\[1.5ex]	
	&\qquad \mathcal{A}_3 = \{ (x, y, m-2, m, z) \mid x \in \{m+1, m+2, \dots, 2m-2\}, \\
	&\qquad \qquad \quad y \in \{x-m, x-m+1, \dots, m-1\} \setminus \{m-2\}, \, z \in \{2m, 2m-1\} \}, \\	
	&\qquad \qquad \qquad \qquad \qquad \quad |\mathcal{A}_3| = 2 \sum_{x=m+1}^{2m-2} (2m-x-1) = (m-2)(m-1); \\[1.5ex]	
	&\qquad \mathcal{A}_4 = \{ (2m-1, 1, m-2, m, 2m), \, (2m-1, m-1, m-2, m, 2m) \}, \\	
	&\qquad \qquad \qquad \qquad \qquad \quad |\mathcal{A}_4| = 2.
\end{flalign*}

In turn, we subdivide the family $\mathcal{B}$ into six disjoint classes:
{\allowdisplaybreaks
\begin{flalign*}
	&\quad \mathcal{B}_1 = \{ (x, y, m, 2m-1, z) \mid x \in \{2, 3, \dots, m-2\}, \\
	&\qquad \qquad y \in \{1, 2, \dots, m+x\} \setminus \{x, z, m\}, \, z \in \{1, m-1\} \}, \\[1ex]	
	&\qquad \qquad \qquad \qquad \qquad |\mathcal{B}_1| = 2\sum_{x=2}^{m-2} (m+x-3) = (m-3)(3m-6); \\[1.5ex]
	&\quad \mathcal{B}_2 = \{ (1, y, m, 2m-1, m-1) \mid y \in \{2, 3, \dots, m+1\} \setminus \{m-1, m\} \}, \\[1ex]
	&\qquad \qquad \qquad \qquad \qquad |\mathcal{B}_2| = m-2; \\[1.5ex]	
	&\quad \mathcal{B}_3 = \{ (m-1, y, m, 2m-1, 1) \mid y \in \{2, 3, \dots, 2m-2\} \setminus \{m-1, m\} \}, \\[1ex]
	&\qquad \qquad \qquad \qquad \qquad |\mathcal{B}_3| = 2m-5; \\[1.5ex]	
	&\quad \mathcal{B}_4 = \{ (m+1, y, m, 2m-1, 1) \mid y \in \{2, 3, \dots, m-1\} \}, \\[1ex]
	&\qquad \qquad \qquad \qquad \qquad |\mathcal{B}_4| = m-2; \\[1.5ex]	
	&\quad \mathcal{B}_5 = \{ (x, y, m, 2m-1, 1) \mid x \in \{m+2, m+3, \dots, 2m-2\}, \\
	&\qquad \qquad y \in \{x-m, x-m+1, \dots, m-1\} \}, \\[1ex]
	&\qquad \qquad \qquad \qquad \qquad |\mathcal{B}_5| = \sum_{x=m+2}^{2m-2} (2m-x) = \frac{1}{2} m(m-3); \\[1.5ex]	
	&\quad \mathcal{B}_6 = \{ (x, y, m, 2m-1, m-1) \mid x \in \{m+1, m+2, \dots, 2m-2\}, \\
	&\qquad \qquad y \in \{x-m, x-m+1, \dots, m-2\} \}, \\[1ex]
	&\qquad \qquad \qquad \qquad \qquad |\mathcal{B}_6| = \sum_{x=m+1}^{2m-2} (2m-x-1) = \frac{1}{2} (m-2)(m-1).
\end{flalign*}
}

These partitions yield the cardinalities of $\mathcal{A}$ and $\mathcal{B}$:
\[
|\mathcal{A}| = (m-3)(3m-8) + 2(2m-5) + (m-2)(m-1) + 2 = 4m^2 - 16m + 18
\]
and
\[
\begin{aligned}	
	|\mathcal{B}| &= (m-3)(3m-6) + (m-2) + (2m-5) + (m-2) \\
	&\quad + \frac{1}{2}m(m-3) + \frac{1}{2}(m-2)(m-1) = 4m^2 - 14m + 10.
\end{aligned}
\]
Whence, we obtain
\[ |\mathcal{B}| - |\mathcal{A}| = (4m^2 - 14m + 10) - (4m^2 - 16m + 18) = 2m - 8>0 \quad \text{for } m \ge 6. \]
By \eqref{eq53}, this strict inequality guarantees that $Y_{m} > Y_{m-2}$.
		
\smallskip	
\noindent \textbf{Case 2.} $r \neq 2$, $\langle n, r \rangle \neq \langle 5, 3 \rangle$, and \eqref{eq13} holds. 
\smallskip

In this case, $m$ is divisible by $4$ and  $m/4 \geq n-4$. 

Let $U$ and $V$ be the vertex sets defined in \eqref{eq15}.	

For each path $F \in \mathcal{A}$ containing distinct vertices $u$ and $v$, 
let $f(F, u, v)$ denote the neighbor of $u$ on the subpath of $F$ connecting $u$ and $v$.

Since $N_{\mathcal{W}_{2m}}(2m)=\{m\}$, the edge $(m, 2m)$ belongs to a path 
from $\mathcal{B}$ only if $r=2$. Thus, for $r \neq 2$, any path $H \in \mathcal{B}$ 
avoids the vertex $2m$ and must include the edge $(m, 2m-1)$, which implies that 
$2m-1$ is not an end-vertex of $H$ (as $r \neq 2$). Consequently, in view of 
$N_{\mathcal{W}_{2m}}(2m-1)=\{1, m-1, m\}$, it follows that $H$ contains exactly 
one of the two subpaths: $(m, 2m-1, 1)$ or $(m, 2m-1, m-1)$.	
\smallskip
		
We decompose each of the families $\mathcal{A}$ and $\mathcal{B}$ into a disjoint union of three subfamilies:
	\begin{align} \label{eq54}
		\mathcal{A} &= \mathcal{A}_1 \sqcup \mathcal{A}_2 \sqcup \mathcal{A}_3, \quad \mathcal{B} = \mathcal{B}_1 \sqcup \mathcal{B}_2 \sqcup \mathcal{B}_3, \\
		\mathcal{A}_1 &= \{ F \in \mathcal{A} \mid 2m \in V(F) \}, \notag \\
		\mathcal{A}_2 &= \{ F \in \mathcal{A} \mid 2m \notin V(F), \ f(F, m-2, m) \notin V \}, \notag \\
		\mathcal{A}_3 &= \{ F \in \mathcal{A} \mid 2m \notin V(F), \ f(F, m-2, m) \in V \}, \notag \\
		\mathcal{B}_1 &= \{ H \in \mathcal{B} \mid m-2 \notin V(H), \ (m, 2m-1, m-1) \subset H \}, \notag \\
		\mathcal{B}_2 &= \{ H \in \mathcal{B} \mid m-2 \notin V(H), \ (m, 2m-1, 1) \subset H \}, \notag \\
		\mathcal{B}_3 &= \{ H \in \mathcal{B} \mid m-2 \in V(H) \}. \notag
	\end{align}
				
Next, we derive the following estimates for the subfamilies of $\mathcal{A}$ and $\mathcal{B}$:
\begin{align} 		
	|\mathcal{A}_1| &\geq |\mathcal{B}_1|, \label{eq55} \\
	|\mathcal{A}_2| &\geq |\mathcal{B}_2|, \label{eq56} \\
	|\mathcal{A}_3| &\geq \frac{m}{4} \binom{m/4}{n-4}, \label{eq57} \\
	|\mathcal{B}_3| &\leq 2n! \binom{2m-4}{n-4}. \label{eq58}
\end{align}

To establish \eqref{eq55}, we construct an injective mapping $g_7 \colon \mathcal{B}_1 \to \mathcal{A}_1$. 
Every $n$-path $H \in \mathcal{B}_1$ admits a representation
\[
H = (x_1, x_2, \dots, x_a, m, 2m-1, m-1, y_1, y_2, \dots, y_{n-a-3}),
\]
where $a \in \{r-1, n-r\}$, all vertices $x_p$ and $y_q$ present in $H$ 
are distinct and belong to $V(\mathcal{W}_{2m}) \setminus \{m-2, m-1, m, 2m-1, 2m\}$. 
The image $g_7(H)$ is defined as
\[
g_7(H) = (x_1, x_2, \dots, x_a, m-2, y_1, y_2, \dots, y_{n-a-3}, m, 2m).
\]

Similarly, to verify \eqref{eq56}, we introduce an injective mapping $g_{8} \colon \mathcal{B}_2 \to \mathcal{A}_2$. 
Any $n$-path $H \in \mathcal{B}_2$ can be represented as
\[
H = (x_1, x_2, \dots, x_a, m, 2m-1, 1, y_1, y_2, \dots, y_{n-a-3}),
\]
where $a \in \{r-1, n-r\}$, the remaining vertices  $x_p$ and $y_q$ (if they exist in $H$) are distinct 
and belong to the set $V(\mathcal{W}_{2m}) \setminus \{1, m-2, m, 2m-1, 2m\}$. 
We define the image $F = g_{8}(H)$ as follows:
\[
F = (x_1, x_2, \dots, x_a, m-2, y_1, y_2, \dots, y_{n-a-3}, m, 2m-1).
\]
It is easy to see that $F \in \mathcal{A}$ and $2m \notin V(F)$. Furthermore, observe that $f(F, m-2, m)$ 
is either $m$ or $y_1$. In the latter case, since $y_1 \neq 2m-1$, the existence 
of the edge $(1, y_1) \in E(H)$ implies that $y_1 \in \{2, 3,\dots, m+1\}$. 
In both cases, $f(F, m-2, m) \notin V$.  Combining these facts, we conclude that $F \in \mathcal{A}_2$.
\smallskip

Next, we obtain a lower bound for $|\mathcal{A}_3|$. For each $v \in V$, with any choice 
of $n-4$ vertices $u_1 < u_2 < \dots < u_{n-4}$ from $U$ (where $|U|=|V|=m/4$), 
we can associate the $n$-path
\[
(u_1, u_2, \dots, u_{r-1}, m-2, v, u_r, u_{r+1}, \dots, u_{n-4}, m, 2m-1) \in \mathcal{A}_3.
\]
This path exists if $n-4 \geq r-1$. For $n \geq 6$ and $r \leq \lceil n/2 \rceil$, 
this condition holds since $n-r \geq \lfloor n/2 \rfloor \geq 3$. For $n \in \{4, 5\}$, 
the only exception $\langle n, r \rangle = \langle 5, 3 \rangle$ does not fall 
under Case 2. Since the total number of such selections is $\frac{m}{4} \binom{m/4}{n-4}$ 
and distinct selections produce distinct paths, we immediately arrive at \eqref{eq57}.
		
\smallskip	
The upper bound for $|\mathcal{B}_3|$ in \eqref{eq58} follows from Proposition~\ref{p2} 
and the fact that every $H \in \mathcal{B}_3$ contains the vertices $m-2$, $m$, $2m-1$, 
and $v$ for some $v \in \{1, m-1\}$.
\smallskip

Finally, the required inequality $Y_{m-2} > Y_{m}$  follows from \eqref{eq53} 
and the relation $|\mathcal{A}| > |\mathcal{B}|$, which in turn is derived from \eqref{eq13} 
and \eqref{eq54}--\eqref{eq58} through the following chain of inequalities:
\begin{align*}		
	|\mathcal{A}| = |\mathcal{A}_1| + |\mathcal{A}_2| + |\mathcal{A}_3| &\geq |\mathcal{B}_1| + |\mathcal{B}_2| + \frac{m}{4} \binom{m/4}{n-4} \\
	&> |\mathcal{B}_1| + |\mathcal{B}_2| + 5n! \binom{2m-4}{n-4} > |\mathcal{B}_1| + |\mathcal{B}_2| + |\mathcal{B}_3| = |\mathcal{B}|. \qedhere
\end{align*}
	\end{proof}
	
\begin{lemma} \label{lemma19}
	If $r \neq 2$ and \eqref{eq24} holds, then $Y_m > Y_{m-3}$.		
\end{lemma}
\begin{proof}
	Let $\mathcal{E}_{17} = \{(m, 2m-2), (m, 2m-1), (m, 2m)\}$. The vertices $m-3$ and $m$ are symmetric in $G_{17} = \mathcal{W}_{2m} \setminus \mathcal{E}_{17}$ since they share the same closed neighborhood: 
	\[
	N_{G_{17}}[m-3] = N_{G_{17}}[m] = \{1, 2, \dots, 2m-2\}=\{1, 2, \dots, 2m-3\}.\] 
Then the relation \eqref{eq1} yields 
	\begin{align} \label{eq59}
		Y_{m} - Y_{m-3} &= |\mathcal{A}| - |\mathcal{B}|, \\
		\mathcal{A} &= \{F \in \mathcal{P}_n(\mathcal W_{2m}, m, r) \mid E(F) \cap \mathcal{E}_{17} \neq \emptyset \}, \notag \\
		\mathcal{B} &= \{H \in \mathcal{P}_n(\mathcal W_{2m}, m-3, r) \mid E(H) \cap \mathcal{E}_{17} \neq \emptyset \}. \notag
	\end{align}

For further analysis of $\mathcal{A}$ and $\mathcal{B}$, we need the following notation.

Let $Q = \{1, m-2, m-1\}$ and let $\mathcal{P}_3(F)$ denote the set of all subpaths of order 3 of a path $F \in \mathcal{A}$. 
The set $\mathcal{S}$ is defined by
\[
\mathcal{S} = \{ (m, 2m-1, 1), (m, 2m-1, m-1), (m, 2m-2, m-2), (m, 2m-2, m-1) \}.
\] 

Using this notation, the cardinalities of $\mathcal{A}$ and $\mathcal{B}$ can be expressed as follows:
\begin{align} \label{eq60}
	\mathcal{A} &= \mathcal{A}_1 \sqcup \mathcal{A}_2 \sqcup \mathcal{A}_3 \sqcup \mathcal{A}_4 \sqcup \mathcal{A}_5, \quad \mathcal{B} = \mathcal{B}_1 \cup \mathcal{B}_2 \cup \mathcal{B}_3 \cup \mathcal{B}_4, \\
	\mathcal{A}_1 &= \{F \in \mathcal{A} \mid \mathcal{P}_3(F) \cap \mathcal{S} = \{ (m, 2m-1, m-1) \} \}, \notag \\
	\mathcal{A}_2 &= \{F \in \mathcal{A} \mid \mathcal{P}_3(F) \cap \mathcal{S} = \{ (m, 2m-2, m-1) \} \}, \notag \\
	\mathcal{A}_3 &= \{F \in \mathcal{A} \mid \mathcal{P}_3(F) \cap \mathcal{S} = \{ (m, 2m-2, m-2) \} \}, \notag \\
	\mathcal{A}_4 &= \{F \in \mathcal{A} \mid \mathcal{P}_3(F) \cap \mathcal{S} = \{ (m, 2m-1, 1) \} \}, \notag \\
	\mathcal{A}_5 &= \mathcal{A} \setminus (\mathcal{A}_1 \cup \mathcal{A}_2 \cup \mathcal{A}_3 \cup \mathcal{A}_4), \notag \\
	\mathcal{B}_1 &= \{H \in \mathcal{B} \mid (m, 2m) \in E(H), \ V(H) \cap Q = \emptyset\}, \notag \\
	\mathcal{B}_2 &= \{H \in \mathcal{B} \mid (m, 2m-1) \in E(H), \ V(H) \cap Q = \emptyset\}, \notag \\
	\mathcal{B}_3 &= \{H \in \mathcal{B} \mid (m, 2m-2) \in E(H), \ V(H) \cap Q = \emptyset\}, \notag \\
	\mathcal{B}_4 &= \{H \in \mathcal{B} \mid V(H) \cap Q \neq \emptyset\}. \notag 
\end{align}

Next, we establish the inequalities below:
	\begin{align}		
		|\mathcal{A}_1| &\geq |\mathcal{B}_1|, \label{eq61} \\
		|\mathcal{A}_2| &\geq |\mathcal{B}_2|, \label{eq62} \\
		|\mathcal{A}_3| &\geq |\mathcal{B}_3|, \label{eq63} \\
		|\mathcal{A}_4| &\geq \binom{m-4}{n-3}, \label{eq64} \\
		|\mathcal{B}_4| &\leq 9n! \binom{2m-4}{n-4}. \label{eq65}
	\end{align}	
	
To establish \eqref{eq61}--\eqref{eq63}, for each $j \in \{1, 2, 3\}$, we construct an injective mapping $f_{j} \colon \mathcal{B}_j \to \mathcal{A}_j$. Here, we set $T = \{z_1, z_2, z_3\}$, where $z_1 = 2m$, $z_2 = 2m-1$, and $z_3 = 2m-2$. Let $H \in \mathcal{B}_j$. Then $V(H) \cap T = \{z_j\}$, and $z_j$ is an end-vertex of $H$. This follows from the conditions $m-3 \in V(H)$, $(m, z_j) \in E(H)$, and $V(H) \cap Q = \emptyset$, together with the neighborhood structures $N_{\mathcal{W}_{2m}}(2m-2)=\{m-2, m-1, m\}$, $N_{\mathcal{W}_{2m}}(2m-1)=\{1, m-1, m\}$, and $N_{\mathcal{W}_{2m}}(2m)=\{m\}$. Let 
\[
V_3 = V(\mathcal{W}_{2m}) \setminus (Q \cup T \cup \{m, m-3\}).
\]
Given this, the path $H$ can be written in the form
\[
H = (x_1, x_2, \dots, x_a, m-3, y_1, y_2, \dots, y_{n-a-3}, m, z_j),
\]
where $a \in \{r-1, n-r\}$, and the vertices $x_p$ and $y_q$ appearing in $H$ are distinct elements of $V_3$. The images $f_{j}(H)$ are defined respectively as follows:
\begin{align*}
	f_{1}(H) &= (x_1, x_2, \dots, x_a, m, 2m-1, m-1, y_1, y_2, \dots, y_{n-a-3} ), \\
	f_{2}(H) &= (x_1, x_2, \dots, x_a, m, 2m-2, m-1, y_1, y_2, \dots, y_{n-a-3}), \\
	f_{3}(H) &= (x_1, x_2, \dots, x_a, m, 2m-2, m-2, y_1, y_2, \dots, y_{n-a-3}).
\end{align*}
	
To prove \eqref{eq64}, for any choice of $n-3$ vertices $u_1 < u_2 < \dots < u_{n-3}$ from the set $\{2, 3, \dots, m-3\}$, we can consider the $n$-path
\[
(u_1, u_2, \dots, u_{r-1}, m, 2m-1, 1, u_r, u_{r+1}, \dots, u_{n-3}) \in \mathcal{A}_4.
\]
Since the number of such selections is exactly $\binom{m-4}{n-3}$ and distinct vertex sets yield distinct \mbox{$n$-paths}, this establishes the desired inequality.

Furthermore, estimate~\eqref{eq65} follows from Proposition~\ref{p2} and the fact that 
every $n$\nobreakdash-path in $\mathcal{B}_4$ contains the vertices $m-3$, $m$, $w_1$, and $w_2$, 
where $w_1 \in \{1, m-2, m-1\}$ and $w_2 \in \{2m-2, 2m-1, 2m\}$.	
		
Combining \eqref{eq24} and \eqref{eq60}--\eqref{eq65}, we obtain:
\begin{align*}		
	|\mathcal{A}| \geq \sum_{i=1}^4 |\mathcal{A}_i| &\geq |\mathcal{B}_1| + |\mathcal{B}_2| + |\mathcal{B}_3| + \binom{m-4}{n-3} \\
	&> |\mathcal{B}_1| + |\mathcal{B}_2| + |\mathcal{B}_3| + 9n! \binom{2m-4}{n-4} \geq \sum_{i=1}^4 |\mathcal{B}_i|=|\mathcal{B}|.
\end{align*}	 	
By \eqref{eq59}, the relation $|\mathcal{A}| > |\mathcal{B}|$ immediately implies that $Y_{m} > Y_{m-3}$.		
\end{proof}
		
\begin{lemma} \label{lemma20}   
	In the graph $\mathcal{W}_{2m}$, the $P_n$-degrees of the vertices $m-1$ and $m$ satisfy the inequality
	\begin{equation*}
		P_n\text{-}\deg_{\mathcal{W}_{2m}}(m) > P_n\text{-}\deg_{\mathcal{W}_{2m}}(m-1).
	\end{equation*}
\end{lemma}

\begin{proof}
	As in the proof of Lemma~\ref{lemma17}, removing the edge $(m, 2m)$ from $\mathcal{W}_{2m}$ yields a graph in which the vertices $m-1$ and $m$ are symmetric. Since their $P_n$-degrees in this modified graph are identical, for the original graph $\mathcal{W}_{2m}$ we have
	\[
	P_n\text{-}\deg_{\mathcal{W}_{2m}}(m) - P_n\text{-}\deg_{\mathcal{W}_{2m}}(m-1) = |\mathcal{C}| - |\mathcal{D}|,
	\]
	where $\mathcal{C}$ and $\mathcal{D}$ are the families of $n$-paths in $\mathcal{W}_{2m}$ containing the edge $(m, 2m)$, such that paths in $\mathcal{D}$ additionally pass through the vertex $m-1$. Clearly, $\mathcal{D} \subseteq \mathcal{C}$. On the other hand, the family $\mathcal{C}$ contains at least one $n$-path outside $\mathcal{D}$, namely
	\[
	(2m, m, 1, 2, \dots, n-2).
	\]
	Therefore, $|\mathcal{C}| > |\mathcal{D}|$, which ensures the desired conclusion.
\end{proof}

\subsection{$k$-condition}

\begin{definition}
	Let $k \geq 4$. A parameter $m$ satisfies the \textit{$k$-condition} if 
	\begin{equation} \label{eq66} 
		\frac{m}{4} \binom{m/4}{k-3} > 3k! \binom{2m-3}{k-3}.
	\end{equation}
\end{definition}

\begin{lemma} \label{lemma21}
	Let $k \geq 4$. There exists a real number $m_0$ such that every $m > m_0$ divisible by $4$ satisfies the $k$-condition.
\end{lemma}

\begin{proof}
	First, we assume that $m \geq 4k-12$ and $m$ is divisible by $4$. Then, since $k \geq 4$, we have $\frac{m}{4} \geq k-3 > 0$ and $2m-3 \geq 8k-27 > k-3 > 0$. This ensures that both binomial coefficients in \eqref{eq66} are well-defined. The difference between the left-hand and right-hand sides of \eqref{eq66} can be expressed as $P(m)$, where the polynomial $P(x)$ is explicitly given by 
	\[
	P(x) = \frac{x}{4(k-3)!} \prod_{i=0}^{k-4} \left(\frac{x}{4} - i\right) - \frac{3k!}{(k-3)!} \prod_{i=0}^{k-4} (2x - 3 - i).
	\]
	
	The polynomial $P(x)$ has degree $k-2$ and a leading coefficient equal to $\dfrac{1}{4^{k-2}(k-3)!}$. Since this leading coefficient is positive, $P(x) \to +\infty$ as $x \to \infty$. It follows that there exists a real number $x_0$ such that $P(x) > 0$ for all $x > x_0$. Let $m_0 = \max\{x_0, 4k-12\}$. Then, for all $m > m_0$ divisible by $4$, we have $P(m) > 0$, and consequently $m$ satisfies the $k$-condition.
\end{proof}

\begin{lemma} \label{lemma22} 
	If $n \geq 4$ and $m$ satisfies the $n$-condition, then
	\begin{enumerate} 
		\item[\textup{(i)}] inequality~\eqref{eq12} holds, which implies $4 \mid m$ and $m/4 \ge n - 3$;
		\item[\textup{(ii)}] $m > 2n$;
		\item[\textup{(iii)}] the inequalities~\eqref{eq8}, \eqref{eq13}, \eqref{eq24}, and \eqref{eq36} are valid.
	\end{enumerate}
\end{lemma}

\begin{proof}
	Suppose that $m$ satisfies the $n$-condition.
	
	\smallskip
	(i) In this case, inequality~\eqref{eq66} rewrites as \eqref{eq12}. The constraints $4 \mid m$ and $m/4 \ge n - 3$ follow immediately from the well-definedness of the binomial coefficient $\binom{m/4}{n-3}$ in \eqref{eq12}.
	
	\smallskip 
	(ii) Assume that $m \leq 2n$. If $n \geq 7$, then $m/4 \leq n/2 < n-3$, which contradicts~(i). For $n \in \{4, 5, 6\}$ and $m \leq 2n$, a direct substitution into \eqref{eq12} also yields a contradiction. Thus, $m > 2n$.
	
	\smallskip
	(iii) Here, we assume that $n \ge 4$ along with the validity of~(i) and~(ii). Under these conditions, all binomial coefficients in \eqref{eq8}, \eqref{eq13}, \eqref{eq24}, and \eqref{eq36} are well-defined. Furthermore, we have $m \ge 12$.
	
\smallskip
First, we establish \eqref{eq36}. By \eqref{eq12}, it suffices to show that
\begin{equation} \label{eq67}
	\binom{m-3}{n-2} > \frac{5m}{12} \binom{m/4}{n-3}.
\end{equation}

Inequality \eqref{eq67} can be written explicitly as
\[
\frac{1}{(n-2)!} \prod_{i=0}^{n-3} (m - 3 - i) > \frac{5m}{12(n-3)!} \prod_{i=0}^{n-4} \left(\frac{m}{4} - i\right),
\]
which leads to the following equivalent form:
\begin{equation} \label{eq68}
	\frac{4^{n-3}}{n - 2} \cdot (m - n) \prod_{i=0}^{n-4} \frac{m - 3 - i}{m - 4i} > \frac{5m}{12}.
\end{equation}

Now, we bound each factor on the left-hand side of~\eqref{eq68} from below.

By Bernoulli's inequality, the first factor satisfies
\begin{equation} \label{eq69}
	\frac{4^{n-3}}{n - 2} = \frac{(1 + 3)^{n-3}}{n - 2} \geq \frac{1 + 3(n - 3)}{n - 2} = \frac{3n - 8}{n - 2} \geq 2.
\end{equation}
	
Since $m > 2n$, it follows that 
\begin{equation} \label{eq70}
	m - n = \frac{m}{2} + \left(\frac{m}{2} - n\right) > \frac{m}{2}.
\end{equation}

For the product term, as each factor with $i \geq 1$ is at least $1$, the condition $m \geq 12$ immediately yields
\begin{equation} \label{eq71}
	\prod_{i=0}^{n-4} \frac{m - 3 - i}{m - 4i} \geq \frac{m-3}{m} \geq \frac{3}{4}.
\end{equation}

Combining \eqref{eq69}--\eqref{eq71}, we obtain \eqref{eq68}:
\[
\frac{4^{n-3}}{n - 2} \cdot (m - n) \prod_{i=0}^{n-4} \frac{m - 3 - i}{m - 4i} > 2 \cdot \frac{m}{2} \cdot \frac{3}{4}
= \frac{3m}{4} > \frac{5m}{12},
\]
thereby ensuring \eqref{eq67} and proving \eqref{eq36}.

Next, using the initial inequalities \eqref{eq12}, \eqref{eq36} and the standard properties of binomial coefficients, we derive \eqref{eq8}, \eqref{eq13}, and \eqref{eq24}.
\smallskip

To begin with, \eqref{eq8} follows straightforwardly from \eqref{eq36}: 
\begin{align*}
	\binom{m-3}{n-2} > 5n! \binom{2m-3}{n-3}= 5n! \binom{2m-4}{n-4} \cdot \frac{2m-3}{n-3} > 2n! \binom{2m-4}{n-4}.
\end{align*}
\smallskip

In turn, relation \eqref{eq12} implies \eqref{eq13}: 
\begin{align*}
	\frac{m}{4} \binom{m/4}{n-4} &= \frac{m}{4} \binom{m/4}{n-3} \cdot \frac{n-3}{m/4 - n + 4}  > 3n! \binom{2m-3}{n-3} \cdot  \frac{n-3}{m/4} \\[1.5ex]
	&=12n! \binom{2m-4}{n-4} \cdot  \frac{2m-3}{n-3} \cdot \frac{n-3}{m} > 5n! \binom{2m-4}{n-4}.
\end{align*}		
\smallskip

Finally, inequality \eqref{eq24} is deduced from \eqref{eq36} through the following steps: 		
\begin{align*}
	\binom{m-4}{n-3} &= \binom{m-3}{n-2} \cdot  \frac{n-2}{m-3} > 5n! \binom{2m-3}{n-3} \cdot \frac{n-2}{m-3}  \\[1.5ex]
	&=5n! \binom{2m-4}{n-4} \cdot \frac{2m-3}{n-3} \cdot \frac{n-2}{m-3} \\[1.5ex]
	& > 5n! \binom{2m-4}{n-4} \cdot  \frac{2m-6}{n-3} \cdot \frac{n-3}{m-3} > 9n! \binom{2m-4}{n-4}. \qedhere
\end{align*}
\end{proof}

\begin{lemma} \label{lemma23}
	Let $k \geq 4$. If $m$ satisfies the $k$-condition, then it also satisfies the $n$-condition for all $4 \leq n \leq k$.
\end{lemma}

\begin{proof}
	Suppose that the $k$-condition holds for $m$. If $k = 4$, the statement is trivial. Thus, we assume $k \geq 5$ and proceed by downward induction on $n$.
	
	The base case $n=k$ is given by the hypothesis. Let $m$ fulfill the $n$-condition for some $5 \leq n \leq k$. Then, in view of Lemma~\ref{lemma22}, we have
	\begin{align*}
		\frac{m}{4} \binom{m/4}{n-4} &= \frac{m}{4} \binom{m/4}{n-3} \cdot \frac{n-3}{m/4 - n + 4} > 3n! \binom{2m-3}{n-3} \cdot \frac{n-3}{m} \\[1.5ex]
		&= 3n! \binom{2m-3}{n-4} \cdot \frac{2m-n+1}{m} > 3(n-1)! \binom{2m-3}{n-4},
	\end{align*}
	which yields
	\[
	\frac{m}{4} \binom{m/4}{n-4} > 3(n-1)! \binom{2m-3}{n-4}.
	\]
	This means that $m$ satisfies the $(n-1)$-condition, completing the inductive step.
\end{proof}

\subsection{Proof of Theorem~\ref{thm1}}	
Fix an integer $k \geq 4$ and let $m$ satisfy the $k$-condition. Consider the graph $\mathcal{W}_{2m}$, any integer $4 \leq n \leq k$, and any root $r$ of the path $P_n$ such that $1 \leq r \leq \lceil n/2 \rceil$. 

By Lemma~\ref{lemma23}, $m$ fulfills the $n$-condition. Consequently, Lemma~\ref{lemma22} ensures that all inequalities required for Lemmas~\ref{lemma9}--\ref{lemma19} hold. According to Lemmas~\ref{lemma9}--\ref{lemma16}, the $(P_n)_r$-degrees of the vertices in $\mathcal{W}_{2m}$ exhibit the chain:
\begin{equation} \label{eq72}
	Y_{m-1} > Y_{m-2} > \dots > Y_1 > Y_{m+1} > Y_{m+2} > \dots > Y_{2m}.
\end{equation}
Furthermore, via Lemmas~\ref{lemma17}--\ref{lemma19}, the position of $Y_m$ relative to the ordering \eqref{eq72} is determined as follows:
\begin{itemize}
	\item $Y_m > Y_{m-1}$ if $r=2$,
	\item $Y_{m-2} < Y_m < Y_{m-1}$ if $\langle n,r \rangle = \langle 5,3 \rangle$,
	\item $Y_{m-3} < Y_m < Y_{m-2}$ if $r \neq 2$ and $\langle n,r \rangle \neq \langle 5,3 \rangle$.
\end{itemize}
In all cases, the $(P_n)_r$-degrees of the vertices in $\mathcal{W}_{2m}$ are pairwise distinct, which establishes that $\mathcal{W}_{2m}$ is $(P_n)_r$-irregular.

Next, Lemma~\ref{lemma20}, relation \eqref{eq72}, and the equality
\[(P_n)\text{-}\deg_{\mathcal{W}_{2m}}(v) = \sum_{r=1}^{\lceil n/2 \rceil} (P_n, r)\text{-}\deg_{\mathcal{W}_{2m}}(v), \quad v \in V(\mathcal{W}_{2m}),\]
together yield the final distribution of the total $P_n$-degrees:
\begin{gather*}
	(P_n)\text{-}\deg_{\mathcal{W}_{2m}}(m) > (P_n)\text{-}\deg_{\mathcal{W}_{2m}}(m-1) > \dots > (P_n)\text{-}\deg_{\mathcal{W}_{2m}}(1) \\
	>(P_n)\text{-}\deg_{\mathcal{W}_{2m}}(m+1) > (P_n)\text{-}\deg_{\mathcal{W}_{2m}}(m+2) 
	> \dots > (P_n)\text{-}\deg_{\mathcal{W}_{2m}}(2m).
\end{gather*}
Hence, the graph $\mathcal{W}_{2m}$ is also $P_n$-irregular.

To conclude the argument, we observe that Lemma~\ref{lemma21} guarantees that the $k$-condition holds for infinitely many values of $m$. Therefore, we obtain an infinite family of graphs~$\mathcal{W}_{2m}$, where $m$ satisfies the $k$-condition, each of which simultaneously exhibits all the prescribed types of irregularity. This completes the proof of Theorem~\ref{thm1}.  
\smallskip 

Theorems~\ref{thm1} and~\ref{thm2} directly imply the following statements.

\begin{corollary} \label{cor1} 
	For every integer $n \geq 3$, there exist infinitely many \mbox{$P_n$-irregular} graphs. 
\end{corollary} 

\begin{corollary} \label{cor2} 
	Except when $n = 3$ and $r$ is the central vertex of $P_3$, there is an infinite number of \mbox{$(P_n)_r$-irregular} graphs for every integer $n \geq 3$ and each root $r$ of $P_n$. 
\end{corollary}

\section{Conclusion}	

This paper initiates the systematic study of $(F)_r$-irregular graphs, a concept originally announced in the monograph~\cite{gtwa}. At the heart of our work is the construction of graphs that unify multiple forms of irregularity. For any $k \geq 4$, we explicitly present infinite families of graphs that are at once $P_n$-irregular and $(P_n)_r$-irregular for each ${n \in \{4, 5, \dots, k\}}$ and all roots $r$ of the path $P_n$. While no nontrivial graph can be $(P_3)_r$-irregular when $r$ is the central vertex of $P_3$, we prove the existence of an infinite class of graphs that exhibit both $P_3$-irregularity and $(P_3)_r$-irregularity when $r$ is an end-vertex.  

The present results confirm Conjecture~\ref{con2} for all paths $P_n$ with $n \geq 3$ (see Corollary~\ref{cor1}).

By analogy with Proposition~\ref{p1}, one can establish a more general limitation on the pairs $\langle F, r \rangle$ for which nontrivial $(F)_r$-irregular graphs exist:

\begin{proposition}
	For any integer $n \geq 3$, if $r$ is the center of the star $K_{1, n-1}$, then no non-trivial graph is $(K_{1, n-1})_r$-irregular.
\end{proposition}

As this appears to be the sole exception, we state the following conjecture.

\begin{conjecture}[Strong Conjecture about $(F)_r$-irregular graphs] \label{con3}
	For every connected graph $F$ of order $|F| \ge 3$ and each root $r$ of $F$, unless $F= K_{1,n-1}$ and $r$ is the center of $K_{1,n-1}$, there exist infinitely many $(F)_r$-irregular graphs.
\end{conjecture}

Corollary~\ref{cor2} establishes the validity of Conjecture~\ref{con3} for all paths $P_n$ with $n \geq 3$. Furthermore, since the rooted and ordinary $C_n$-degrees of a vertex coincide, Conjecture~\ref{con3} also holds for cycles by~\cite{r4}. Similarly, this result extends to directed cycles (investigated as ``oriented cycles'' in~\cite{r7}). 
In conclusion, we raise the following open problem:

\smallskip
\noindent \textbf{Problem.} \textit{Determine whether analogues of Theorems~\ref{thm1} and~\ref{thm2} hold for any directed path of order $n \ge 3$.}

\end{document}